\documentclass[11pt]{amsart}
\usepackage{url,enumitem, amsopn}

\newcommand{\cal}{\mathcal}

\newcommand{\reals}{\mbox{$\mathbb R$}}

\newcommand{\nats}{\mbox{$\mathbb N$}}
\newcommand{\ints}{\mbox{$\mathbb Z$}}

\newcommand{\spn}{{\mathop{\mathrm{Span}}\nolimits}}

\newcommand{\soc}{{\mathop{\mathrm{Soc}}\nolimits}}
\newcommand{\doc}{{\mathop{\partial\mathrm{oc}}\nolimits}}
\newcommand{\m}{{\mathop{\mathbf{m}}\nolimits}}
\newcommand{\p}{{\mathop{\mathbf{p}}\nolimits}}
\newcommand{\q}{{\mathop{\mathbf{q}}\nolimits}}
\newcommand{\comment}[1]{}

\DeclareRobustCommand{\stirling}{\genfrac\{\}{0pt}{}} 

\usepackage{amsmath, amsfonts, amssymb, latexsym, stmaryrd}
\usepackage{graphicx}                   

\def\squarebox#1{\hbox to #1{\hfill\vbox to #1{\vfill}}}
\def\qed{\hspace*{\fill}
        \vbox{\hrule\hbox{\vrule\squarebox{.667em}\vrule}\hrule}\smallskip}

\theoremstyle{plain}
\newtheorem{lemma}{Lemma}[section]
\newtheorem{theorem}[lemma]{Theorem}
\newtheorem{corollary}[lemma]{Corollary}
\newtheorem{proposition}[lemma]{Proposition}
\theoremstyle{definition}
\newtheorem{claim}[lemma]{Claim}
\newtheorem{observation}[lemma]{Observation}
\newtheorem{definition}[lemma]{Definition}

\newtheorem{question}[lemma]{Question}

\newtheorem*{rmk*}{Remark}
\newtheorem*{rmks*}{Remarks}
\newtheorem*{conventions*}{Conventions}
\newtheorem*{convention*}{Convention}
\newtheorem*{example*}{Example}

\def\squareforqed{\hbox{\rlap{$\sqcap$}$\sqcup$}}
\def\qed{\ifmmode\squareforqed\else{\unskip\nobreak\hfil
\penalty50\hskip1em\null\nobreak\hfil\squareforqed
\parfillskip=0pt\finalhyphendemerits=0\endgraf}\fi}

\newlength{\tablength}
\newlength{\spacelength}
\newcommand{\tabstar}{\hspace*{\tablength}}
\newcommand{\spacestar}{\hspace*{\spacelength}}
\def\obeytabs{\catcode`\^^I=\active}
{\obeytabs\global\let^^I=\tabstar}
{\obeyspaces\global\let =\spacestar}
\newenvironment{display}{\begingroup\obeylines\obeyspaces\obeytabs}{\endgroup}
\newenvironment{prog}{\begin{display}\parskip0pt\sf}{\end{display}}

\author{Geir Agnarsson}
\address{Department of Mathematical Sciences \\ George Mason University \\ Fairfax, VA  22030}
\email{geir@math.gmu.edu}

\author{Neil Epstein}
\address{Department of Mathematical Sciences \\ George Mason University \\ Fairfax, VA  22030}
\email{nepstei2@gmu.edu}

\title{On monomial ideals and their socles}

\date{\today}

\begin{document}

\begin{abstract}
For a finite subset $M\subset [x_1,\ldots,x_d]$ of monomials,
we describe how to constructively obtain a monomial ideal 
$I\subseteq R = K[x_1,\ldots,x_d]$ such that the set of monomials
in $\soc(I)\setminus I$ is precisely $M$, or
such that $\overline{M}\subseteq R/I$ is a $K$-basis for the
the socle of $R/I$. For a given $M$ we obtain a natural class of monomials $I$ 
with this property. This is done by using solely the 
lattice structure of the monoid $[x_1,\ldots,x_d]$. 
We then present some duality results by using anti-isomorphisms between 
upsets and downsets of $({\ints}^d,\preceq)$.
Finally, we define and analyze zero-dimensional monomial ideals of $R$ 
of type $k$, where type $1$ are exactly the Artinian Gorenstein ideals,
and describe the structure of such ideals that correspond to order-generic 
antichains in ${\ints}^d$.
\end{abstract}




\maketitle

\section{Introduction}
\label{sec:intro}

Aside from the algebraic benefits of the socle of a module
over a commutative ring, when studying local rings, Cohen-Macaulay
rings or Gorenstein rings~\cite{Bruns-Herzog,Eisenbud,MonAlg}
the socle of an ideal, in particular that of a monomial ideal 
of the polynomial ring over a field, has a rich
lattice structure that many times can be studied solely
by using combinatorial arguments. Consider the following:
any range of mountains on a given piece of land where each 
mountain has the same cone-type shape is, of course,
determined by the loci of the mountain tops. But is it 
determined by the mere valleys these mountains form?
The answer depends on what additional 
conditions we know the mountain range has. This is one of the 
motivating questions in this article when we view a monomial ideal as a 
mountain range and the valleys it forms as the elements 
in the socle of the ideal that are not in the ideal.

For a field $K$, monomial ideals
of the polynomial ring $R = K[x_1,\ldots,x_d]$ play a pivotal role
in the investigation of Gr\"{o}bner bases of general ideals of
$R$~\cite{Becker-etal,Loust}. The underlying reason as to
why monomial ideals are so convenient, yet so important, in 
investigating reduction systems for general ideals of $R$ stems
in part from the fact that for every monomial ideal $I = (m_1,\ldots,m_k)$,
where no $m_i$ divides another $m_j$, has a Gr\"{o}bner basis
$M = \{m_1,\ldots,m_k\}$; exactly the minimum set of generators for $I$.
For $A\subseteq \{1,\ldots,k\}$ let $m_A$ denote the least common
multiple of the $m_i : i\in A$. With this convention one can
define the {\em Scarf complex} $\Delta(I)$ of the monomial ideal $I$
as the simplicial complex consisting of all the subsets
$A\subseteq \{1,\ldots,k\}$ with unique least common multiple $m_A$,
that is
\[
\Delta(I) = 
\{ A\subseteq \{1,\ldots,k\} : m_A \neq m_B \mbox{ for all } B\neq A\}.
\]
The Scarf complex discussed in~\cite{Bayer-Peeva-Sturmfels}
and~\cite{Miller-Sturmfels} was first introduced in~\cite{HerbertScarf}.
It is easy to see that the facets $F_{d-1}(\Delta(I))$ of the Scarf complex
$\Delta(I)$ are in bijective correspondence with the maximal monomials
of $R\setminus I$ (w.r.t.~the divisibility partial order), which is exactly
the set of monomials of the socle $\soc(I)$ that
are not in $I$ (see definition in the following section.) 
The cardinality of this
very set of monomials has many interesting combinatorial interpretations,
two of which we will briefly describe here below.

We say that an ideal $I$ of $R$ is {\em co-generated} by a set
${\cal{F}}$ of $K$-linear functionals $R\rightarrow K$ if $I$ is
the largest ideal of $R$ contained in all the kernels of the functionals
in ${\cal{F}}$. A celebrated result by Macaulay from
1916~\cite{Macaulay} states that every ideal of $R$ is finitely co-generated,
so $R$ is {\em co-Noetherian} in this sense. More specifically, it turns
out that any monomial ideal $I$ of $R$ has at least $|F_{d-1}(\Delta(I))|$
co-generators and can always be co-generated by
$|F_{d-1}(\Delta(I))| + 1$ functionals (see~\cite{A-co-gen}). 
Further, for a given ideal 
$I$ (not necessarily monomial) of $R$ and a fixed term order,
then $I$ has a Gr\"{o}bner basis where the head or leading terms
of the basis elements form a corresponding monomial ideal $L(I)$ of $R$,
and $I$ can then be co-generated by one functional if
$|F_{n-1}(\Delta(L(I)))| < d$ and by $|F_{n-1}(\Delta(L(I)))| + 1$ functionals
otherwise~\cite{A-co-gen}.

The cardinality $|F_{d-1}(\Delta(I))|$ is also linked to the number
of edges in a simple graph on $k$ vertices in the following way.
For given $d,k\in\nats$ let $c_d(k)$ denote the maximum number
of facets $|F_{d-1}(\Delta(I))|$ among all monomial ideals $I$ of
$R$ that are minimally generated by $k$ monomials.
In general, the Scarf complex $\Delta(I)$ is always
a sub-complex of the boundary complex of a simplicial
polytope $P(I)$ on $k$ vertices where one facet is missing.
When $I$ is Artinian (or zero-dimensional) 
and generic (i.e.~the powers of all $x_i$ in
the generators for $I$ are distinct) then $\Delta(I)$ actually equals this
boundary complex $P(I)$ with one facet missing~\cite{Miller-Sturmfels}.
Using the Dehn-Sommerville equations for the simplicial polytope
$P(I)$ one can show that determining $c_d(k)$ for $d\leq 5$
is equivalent to determining the maximum number of edges of the
$1$-skeleton of $P(I)$. This turns out to be the same problem
of determining the maximum number of a simple graph on $k$ vertices
of order dimension at most $d$~\cite{A-outside}.

We therefore see that the number of monomials in $\soc(I)\setminus I$,
the socle of the monomial ideal $I$ that are not in $I$, has some
interesting combinatorial interpretations in addition to the known
algebraic ones, in particular, the maximum number that the set
$\soc(I)\setminus I$ can have. Trivially we have $c_2(k) = k-1$ 
for any $k\in\nats$ and less trivially we know that 
$c_3(k) = 2k-5$ for $k\geq 3$~\cite{A-outside,Miller-Sturmfels}.
For each $d\geq 4$ the exact value of $c_d(k)$ is still unknown, although
the asymptotic behavior does satisfy $c_d(k) = \Theta(k^{\lfloor d/2\rfloor})$ 
as $k$ tends to infinity and $d$ is fixed~\cite{A-outside}. In any case,
we see from this that a monomial ideal minimally generated by $k$ monomials,
can ``generate'' monomials in its socle of cardinality
considerably larger order than that of $k$. In order for that to occur though,
the $k$ monomials that generate the given ideal must be special and relate
to each other in a singular way. By the same token, the resulting monomials in
$\soc(I)\setminus I$ also relate in a special way if they are generated
by ``few'' generating monomials of $I$. Looking away from these
extreme cases for a moment, some natural questions about monomial
ideals and their socles, in particular about those monomials in the
socle and not in the ideal, arise.
\begin{enumerate}[label=(\roman*)]
\item Assume that for a given (apriori
unknown) monomial ideal $I$ the set $\soc(I)\setminus I$ is given,
is it always possible to retrieve the ideal $I$ from it, if not, what 
additional structure is needed? In which cases
is it unique?
\item Given any set of monomials that form an antichain w.r.t.~divisibility
order, can one always find a monomial ideal $I$ such that this set
of monomials has the form $\soc(I)\setminus I$? Is the ideal $I$ unique?
If not, can they be characterized in some way?
\end{enumerate}
The purpose of this
article is to address these questions, discuss uniqueness and
non-uniqueness, present up-down duality results, discuss generalizations to
general ideals of $R$ and address what is similar with the monomial
ideal case and what is not.


This article was in part inspired by the M.S.~thesis of
Anna-Rose Wolff~\cite{Wol16}.
She analyzed the \emph{survival complex} $\Sigma(R/I)$ of $R/I$ where
$I$ is a monomial ideal that contains powers of all the variables of $R$.
The survival complex is a simplicial complex whose vertices are the
monomials of $R$ that are not in $I$, where a simplex consists of a
set of monomials whose product is not in $I$.  She showed
\cite[Prop.~2.2.1]{Wol16} that the truly isolated points
of that complex correspond to the monomial basis of the socle of
$I$, a set which in this document we call $\doc(I)$
(see Definition~\ref{def:doc(I)}.) As such,
it was natural for her to address the question we address here,
completely solving the two-variable case
\cite[Algorithm 3.3.1 and Proposition 3.3.3]{Wol16}
and making headway on the three-variable case
\cite[Algorithm 3.4.1]{Wol16}.

The rest of this article is organized as follows:

In Section~\ref{sec:posets} we define some basic concepts about
partially ordered sets that we will be using and referring to throughout the
article.

In Section~\ref{sec:socle-ideal} we show that if for a monomial
ideal $I$ the set $\soc(I)\setminus I$ is given, then we can
retrieve a monomial ideal $I$ back from a derived version of the set
$\soc(I)\setminus I$.

In Section~\ref{sec:up-down} we present and use an intuitively obvious 
duality between monomial ideals of $R$ on the one hand and
``down-sets'' of monomials on the other hand, 
that contain all elements dividing a given element of the down-set 
(see definitions in Section~\ref{sec:up-down}).

In Section~\ref{sec:zero-monomial} we elaborate more explicitly the discussion
presented in Section~\ref{sec:socle-ideal} and show that if
$\soc(I)\setminus I$ is given, then there is a unique zero-dimensional
ideal $I$ yielding the given set of monomials. As a corollary we
obtain a well-known description of Artinian monomial
ideals of $R$ that are Gorenstein. We then use what we have established
to describe Artinian monomial ideals that are {\em type 2 Gorenstein} (where
{\em type 1 Gorenstein} is the usual Gorenstein notion (see definitions in
Section~\ref{sec:zero-monomial})).

In Section~\ref{sec:kgeq3} we discuss further Artinian
{\em type $k$ Gorenstein} monomial ideals and describe their 
{\em order-generic} case.

In Section~\ref{sec:general} we discuss the socle of ideals of
$R = K[x_1,\ldots,x_d]$ in more general settings when $K$ is an algebraically
closed field.

Finally, in Section~\ref{sec:some-q} we summarize our main results
and pose some questions.

\section{Partially ordered sets, basic definitions and notations}
\label{sec:posets}

The set of integers will be denoted by $\ints$, the
set of natural numbers by $\nats$ and the set of
non-negative numbers $\nats\cup\{0\}$ by $\nats_0$. For each
$n\in\nats$ we let $[n] := \{1,\ldots,n\}$. 

Throughout, $R = K[x_1,\ldots,x_d]$ will denote the polynomial ring
over a field $K$. By the {\em socle} of an ideal $I\subseteq R$
w.r.t.~the maximal ideal $\m = (x_1,\ldots,x_d)$ of $R$, we will mean the ideal 
\[
\soc(I) = \soc_{\m}(I) := (I:\m) = \{f\in R: x_if\in I \mbox{ for each }i\}.
\]
Note that for a monomial ideal $I\subseteq R$, the set
$\soc(I)\setminus I$ contains monomials $a$ such that
(i) $a\not\in I$ and (ii) $x_ia\in I$ for every $i\in[d]$.
\begin{definition}
\label{def:doc(I)}
For a monomial ideal $I$ and the maximal ideal $\m = (x_1,\ldots,x_d)$ 
of $R$, let $\doc(I)$ denote the set of the monimials in
$\soc(I)\setminus I$.
\end{definition}
\newpage
\begin{rmks*}\ 
\begin{enumerate}[label=(\roman*)]
\item By the above Definition~\ref{def:doc(I)}, we have
$\soc(I) = I\oplus\spn_K(\doc(I))$ as a $K$-vector spaces.
\item The image $\overline{\soc(I)} = (\overline{0}:\m/I) = \soc(R/I)$
in $R/I$ has a basis as a vector space over $K$ consisting of the images of the
monomials in $\doc(I)$ (see ~\cite{MonAlg,RingsCatMod}.)
Hence, $\doc(I)$ consists exactly of the maximal monomials in
$R\setminus I$ w.r.t.~the divisibility partial order and
$|\doc(I)| = \dim_K(\soc(R/I))$.
\end{enumerate}
\end{rmks*}

Assume for a moment we have a general ideal $I$ and a general maximal 
ideal $\m$ of $R$. If $\m$ does not contain $I$, then $\m+I = R$ and hence for
$f\in (I:\m)$ we have $f \in fI + f\m \subseteq I$. Therefore 
$(I:\m) = I$ and so $\soc_{\m}(I)$ is trivial.
Therefore, {\em for an ideal $I$ and a maximal ideal $\m$ of $R$, the socle
$\soc_{\m}(I)$ is only of interest when $I\subseteq \m$.}
Hence, unless otherwise stated, we will always assume $I\subseteq \m$.

Most of what we establish from now on about monomial ideals
of $R = K[x_1,\ldots,x_d]$ will only use the the monoid
$[x_1,\ldots,x_d]$ as a partially ordered set (poset)
with the partial order given by divisibility.
We present some basic definitions and notations for general posets
that we will use throughout the article.

For a poset $(P,\leq)$ recall that $C\subseteq P$
is a {\em chain} if $C$ forms a linearly or totally ordered
set within $(P,\leq)$. A subset $N\subseteq P$ is an {\em antichain}
or a {\em clutter} if no two elements in $N$ are comparable
in $(P,\leq)$.

For a poset $(P,\leq)$ call a subset $U\subseteq P$ an {\em upset},
or an {\em up-filter}, of $P$ if $x\geq u\in U \Rightarrow x\in U$.
Call a subset $D\subseteq P$ a {\em downset},
or a {\em down-filter}, of $P$ if if $x\leq d\in D \Rightarrow x\in D$.

For an upset $U$ of a poset $P$ let $S_d(U) := \max(P\setminus U)$
be the maximal elements of $P\setminus U$. 
For a downset $D$ of a poset $P$ let $S_u(D) := \min(P\setminus D)$
be the minimal elements of $P\setminus D$.

For a subset $A\subseteq P$,
let $U(A) := \{x\in P : x\geq a \mbox{ for some } a\in A\}$
be the {\em upset generated by $A$}, and 
$D(A) := \{x\in P : x\leq a \mbox{ for some } a\in A\}$
the {\em downset generated by $A$}. If $A = \{a_1,\ldots,a_n\}$
is finite, then we write $U(a_1,\ldots,a_n)$ (resp.~$D(a_1,\ldots,a_n)$)
for $U(A)$ (resp.~$D(A)$.) 

Recall that ${\ints}^d$ has a natural partial order $\preceq$
where for $\tilde{a} = (a_1,\ldots,a_d)$ and $\tilde{b} = (b_1,\ldots,b_d)$
we have $\tilde{a}\preceq\tilde{b} \Leftrightarrow a_i\leq b_i$ for
each $i\in [d]$. This is a generalization of the partial order
given in~\cite[Ex.~19, p.~71]{DiscMath}, and we will henceforth
be using ``$\leq$'' for this partial order ``$\preceq$''. 
The usual componentwise addition $+$, which makes $({\ints}^d,+)$ an
abelian group, respects the mention partial order $\leq$.

For any fixed point $\tilde{c}\in {\ints}^d$, the map
$\tau_{\tilde{c}} : {\ints}^d \rightarrow {\ints}^d$ given by 
$\tau_{\tilde{c}}(\tilde{x}) = \tilde{c} + \tilde{x}$ is an 
order preserving translation, and hence an order automorphism,
of the poset $({\ints}^d, \leq)$. The following is easy to obtain.
\begin{observation}
\label{obs:translation}
For any $A\subseteq {\ints}^d$ and $\tilde{c}\in {\ints}^d$
we have
\[
\tau_{\tilde{c}}(U(A)) = U(\tau_{\tilde{c}}(A)), \ \ 
\tau_{\tilde{c}}(D(A)) = D(\tau_{\tilde{c}}(A)).
\]
Further, for any upset $U\subseteq {\ints}^d$ and downset 
$D\subseteq {\ints}^d$ we have
\[
\tau_{\tilde{c}}(S_d(U)) = S_d(\tau_{\tilde{c}}(U)), \ \ 
\tau_{\tilde{c}}(S_u(D)) = S_u(\tau_{\tilde{c}}(D)).
\]
\end{observation}
Writing out what the above observation states in terms of $\tilde{c}$ we get
\begin{gather*}
U(A+\tilde{c}) = U(A) + \tilde{c}, \quad D(A+\tilde{c}) = D(A) + \tilde{c},\\
S_d(U + \tilde{c}) = S_d(U) + \tilde{c}, \quad S_u(D + \tilde{c}) = S_u(D) + \tilde{c}.
\end{gather*}

An upset $U\subseteq {\ints}^d$ is {\em cofinite}
if here is a $\tilde{c}\in {\ints}^d$ such that
(i) $U\subseteq U(\tilde{c})$ and (ii) $U(\tilde{c})\setminus U$
is finite. Likewise a downset $D\subseteq {\ints}^d$ is
{\em cofinite} if there is a $\tilde{c}\in{\ints}^d$
such that (i) $D\subseteq D(\tilde{c})$ and (ii) $D(\tilde{c})\setminus D$
is finite.

The map given by
$\tilde{p} = (p_1,\ldots,p_d)\mapsto x_1^{p_1}\cdots x_d^{p_d}
= {\tilde{x}}^{\tilde{p}}$ is an isomorphism between
the additive monoid ${\nats}_0^d$ and the multiplicative
monoid $[x_1,\ldots,x_d]$. Moreover, the map is also
an order isomorphism from the poset $({\nats}_0^d,\leq)$, where
$\leq$ is the partial order of ${\ints}^d$ from above,
to $[x_1,\ldots,x_d]$ as a poset given by divisibility.
Hence as the map is an order preserving monoid isomorphism,
we will for the most part in the next two sections 
deal with the additive monoid ${\nats}_0^d$
and the abelian group ${\ints}^d$ instead of the multiplicative
monoid $[x_1,\ldots,x_d]$ within the ring $R$. Via this isomorphism
we have a bijective correspondence between monomial ideals of $R$ and
upsets of the poset ${\nats}_0^d$.

\section{Retrieving an upset from maximal elements not in the upset}
\label{sec:socle-ideal}

We will in this section and the following section 
describe everything in terms of ${\nats}_0^d$ or ${\ints}^d$ 
instead of the ring $R = K[x_1,\ldots,x_d]$ as we justified in
the last paragraph in the previous section.

Let $G = \{\tilde{p}_1,\ldots,\tilde{p}_k\}\subseteq {\nats}^d$
be an antichain.
For the moment we assume that $G$ is {\em positive}, 
i.e.~all the coordinates of each $\tilde{p}_i$ are strictly positive. 
(This is by no means a restrictive assumption as will become clear
as we go along.)
If $\tilde{p}_i = (p_{i\/1},\ldots,p_{i\/d})$ for each $i$, then
for each $j\in[d]$ let 
\begin{equation}
\label{eqn:mi-Mi}
m_j = \min\{ p_{i\/j} : 1\leq i\leq d\} - 1, \ \ 
M_j = \max\{ p_{i\/j} : 1\leq i\leq d\} + 1
\end{equation}
and 
\begin{eqnarray}
\label{eqn:B^*}
\lefteqn{B^*(G) = } \nonumber \\
& & \{(M_1,m_2,m_3\ldots,m_d), (m_1,M_2,m_3,\ldots,m_d),
(m_1,m_2,M_3,\ldots,m_d), \nonumber \\
& & \ldots,(m_1,m_2,m_3,\ldots,M_d)\}. 
\end{eqnarray}
Let $G^* = G \cup B^*(G)$. That $\tilde{r}\in S_d(U(G))$,
means precisely that $\tilde{r}\not\in U(G)$ and 
$\tilde{r}+\tilde{e}_j\in U(G)$ for each $j$,
where $\tilde{e}_j$ is the usual basis vector for ${\reals}^d$
with $1$ in the $j$-coordinate and $0$ everywhere else.
Since $\tilde{r}\not\in U(G)$ and $\tilde{r}+\tilde{e}_j\in U(G)$ 
for each given $j$, then for some $\tilde{p}_i\in G$ we have
$r_j = p_{i\/j} - 1 < M_j$. Since this
holds for each $j$, then $\tilde{r}$ cannot be greater
than any element of $B^*(G)$ and so $\tilde{r}\not\in U(B^*(G))$
and hence $\tilde{r}\not\in U(G^*)$. Since 
$G\subseteq G^*$ and so $U(G)\subseteq U(G^*)$,
we have $\tilde{r}+\tilde{e}_j\in U(G^*)$ for each $j$,
and so $\tilde{r}\in S_d(U(G^*))$.

Also we have $U(G)\subseteq U(G^*)\subseteq U(\tilde{m})$,
where $\tilde{m} = (m_1,\ldots,m_d)$ from (\ref{eqn:mi-Mi}),
and $U(\tilde{m})\setminus U(G^*)\subseteq U(\tilde{m})\setminus U(B^*(G))$,
which is finite. We summarize in the following.
\begin{proposition}
\label{prp:G^*}
For any antichain 
$G = \{\tilde{p}_1,\ldots,\tilde{p}_k\}\subseteq {\nats}^d$ we have
\[
S_d(U(G))\subseteq S_d(U(G^*))\subseteq U(\tilde{m})\setminus U(G^*),
\]
which is a finite set, and hence $U(G^*)$ is cofinite.
\end{proposition}
Note that the $S_d(U(G^*))$ corresponds to the socle of the
monomial ideal whose generators correspond to $G^*$. The main
result in this section is the following.
\begin{theorem}
\label{thm:main-gen}
For a given antichain   
$G = \{\tilde{p}_1,\ldots,\tilde{p}_k\}\subseteq {\nats}^d$ we
have 
\[
G = S_u(D(S_d(U(G^*)))).
\]
\end{theorem}
Before proving Theorem~\ref{thm:main-gen} we need the following.
\begin{lemma}
\label{lmm:disjoint}
Let $(P,\leq)$ be a poset. For an upset $U\subseteq P$ and
a downset $D\subseteq P$ we have 
\[
D(S_d(U))\cap U = U(S_u(D))\cap D = \emptyset.
\]
\end{lemma}
\begin{proof}
If $\tilde{x}\in D(S_u(U))$, 
then $\tilde{x}\leq\tilde{y}\in S_d(U)$.
If now $\tilde{x}\in U$ as well, then by definition of an upset we must
have $\tilde{y}\in U$, which is a contradiction since 
$U\cap S_d(U) = \emptyset$. In the same way we obtain 
$U(S_u(D))\cap D = \emptyset$.
\end{proof}
\newpage
\begin{conventions*}\ 
\begin{enumerate}[label=(\roman*)]
\item For a set $A\subseteq {\ints}^d$ and ${i}\in[d]$ 
let $A_{x_{i} = a_{i}} := \{\tilde{x}\in A : x_{i}=a_{i}\}$.
\item For a finite set $G\subseteq {\ints}^d$ 
let $\overline{U}(G) = U(\tilde{m})\setminus U(G^*)$,
where $\tilde{m}$ is as in (\ref{eqn:mi-Mi}), which then is a finite set.
\item For ${\cal{I}} = \{i_1,\ldots,i_h\}\subseteq [d]$
let $\pi_{\cal{I}} = \pi_{i_1,\ldots,i_h} : {\reals}^d \rightarrow {\reals}^h$
denote the natural projection onto coordinates $i_1,\ldots,i_h$.
In particular $\pi_{i}(\tilde{x}) = x_{i}$ and 
$\pi_{\hat{\imath}}(\tilde{x}) = (x_1,\ldots,x_{{i}-1},x_{{i}+1},\ldots,x_d)$.
\end{enumerate}
\end{conventions*}
\begin{proof}[Proof of Theorem~\ref{thm:main-gen}]
Let $\tilde{p} = (p_1,\ldots,p_d) \in G$.
Since $U(G^*)$ is cofinite, then for each ${i}\in[d]$
the set $\overline{U}(G)_{x_{i} = p_{i}-1}$ is also a finite
poset. Also, it is nonempty since
$\tilde{p}-\tilde{e}_i = (p_1,\ldots,p_{{i}-1},p_{i}-1,p_{{i}+1},\ldots,p_d)\in 
\overline{U}(G)_{x_{i} = p_{i}-1}$ as $G^*$ is an antichain. 
Hence, there is a maximal
element $\tilde{q}\in \overline{U}(G)_{x_{i} = p_{i}-1}$ with
$\tilde{p}-\tilde{e}_{i}\leq\tilde{q}$. Since 
$\tilde{q} + \tilde{e}_{i} \geq \tilde{p}$, then
$\tilde{q} + \tilde{e}_{i}\in U(\tilde{p})\subseteq U(G)\subseteq U(G^*)$.
Also, since $\tilde{q}$ is maximal in $\overline{U}(G)_{x_{i} = p_{i}-1}$,
we have $\tilde{q}+\tilde{e}_{\ell}\in U(G^*)_{x_{i} = p_{i}-1}\subseteq U(G^*)$
for all ${\ell}\neq {i}$, and so $\tilde{q}\in S_d(U(G^*))$.
So, for each ${i}\in[d]$ we have 
$\tilde{p} - \tilde{e}_{i} \in D(S_d(U(G^*)))$. Since 
$\tilde{p}\in G\subseteq U(G^*)$, then by Lemma~\ref{lmm:disjoint}
we have $\tilde{p}\not\in D(S_d(U(G^*)))$, and therefore
$\tilde{p}\in S_u(D(S_d(U(G^*))))$. This proves that
$G\subseteq S_u(D(S_d(U(G^*))))$.

For the other direction, we first verify that
\begin{equation}
\label{eqn:m+1}
S_u(D(S_d(U(G^*))))\subseteq U(\tilde{m}+\tilde{1}),
\end{equation}
where $\tilde{1} = (1,\ldots,1)\in {\ints}^d$.
First note that by Proposition~\ref{prp:G^*} we have 
$S_d(U(G^*))\subseteq U(\tilde{m})\setminus U(G^*)\subseteq U(\tilde{m})$.

If $\tilde{r}\in S_u(D(S_d(U(G^*))))$, then 
$\tilde{r}-\tilde{e}_{i}\in D(S_d(U(G^*)))$ for each ${i}$, 
and so $\tilde{r}-\tilde{e}_{i}\leq \tilde{q}$ for some 
$\tilde{q}\in S_d(U(G^*))\subseteq U(\tilde{m})$ for each ${i}$.
Consider a fixed ${i}$. Since $\tilde{r}\not\leq \tilde{q}$ 
and $\tilde{r}-\tilde{e}_{i}\leq \tilde{q}$, we must have
$r_{i}-1=q_{i}$. Since $\tilde{q}\in S_d(U(G^*))\subseteq U(\tilde{m})$,
we have $q_{i}\geq m_{i}$ and therefore 
$r_{i} = q_{i}+1 \geq m_{i} +1$. Since this holds
for each ${i}$, we have thus (\ref{eqn:m+1}).

For $\tilde{r}\in S_u(D(S_d(U(G^*))))$ we have for each ${i}$ that
$\tilde{r}-\tilde{e}_{i}\in D(S_d(U(G^*)))$, 
and hence, by Lemma~\ref{lmm:disjoint}, 
$\tilde{r}-\tilde{e}_{i}\not\in U(G^*)$ for each ${i}$.
If now $\tilde{r}\in U(G^*)$, then $\tilde{r}$ is a minimal 
element of $U(G^*)$ and hence $\tilde{r}\in G^*$. 
Hence
\[
\tilde{r} \in G^*\cap S_u(D(S_d(U(G^*))))\subseteq 
G^* \cap U(\tilde{m}+\tilde{1}) = G.
\]
Therefore it suffices to show that $S_u(D(S_d(U(G^*))))\subseteq U(G^*)$.

Assume $\tilde{r}\in S_u(D(S_d(U(G^*))))\setminus U(G^*)$.
We then have 
$\tilde{r}\in U(\tilde{m}+\tilde{1})\setminus U(G^*)\subseteq 
U(\tilde{m})\setminus U(G^*)$.
By Proposition~\ref{prp:G^*} $U(G^*)$ is cofinite, and so
there is a maximal element 
$\tilde{q}\in U(\tilde{m})\setminus U(G^*)$ with $\tilde{r}\leq \tilde{q}$.
In this case we have, by the definition of $S_d$, that
$\tilde{r}\leq\tilde{q}\in S_d(U(G^*))$
and hence $\tilde{r}\in D(S_d(U(G^*)))$ contradicting
that $\tilde{r}\in S_u(D(S_d(U(G^*))))$. This proves that 
$S_u(D(S_d(U(G^*))))\subseteq U(G^*)$, and hence, by previous paragraph, 
$S_u(D(S_d(U(G^*))))\subseteq G$.
\end{proof}
\begin{rmk*} The values $m_1,\ldots,m_d$ and $M_1,\ldots,M_d$
from (\ref{eqn:mi-Mi}), used to define $G^*$,
do not play a major role, except for merely being small and respectively 
large enough. In fact, if 
$\tilde{a}\leq \tilde{m}$ and $\tilde{b}\geq \tilde{M}$ and  
\begin{equation}
\label{eqn:B^*(a,b)}
B^*(\tilde{a},\tilde{b}) = \{(b_1,a_2,a_3\ldots,a_d), (a_1,b_2,a_3,\ldots,a_d),
(a_1,a_2,b_3,\ldots,a_d),\ldots,(a_1,a_2,a_3,\ldots,b_d)\},
\end{equation}
then we can define 
$G^*(\tilde{a},\tilde{b}) = G \cup B^*(\tilde{a},\tilde{b})$,
and we then, as in Proposition~\ref{prp:G^*}, obtain that  
\[
S_d(U(G))\subseteq S_d(U(G^*(\tilde{a},\tilde{b})))
\subseteq U(\tilde{a})\setminus U(G^*(\tilde{a},\tilde{b})),
\]
which is a finite set, and hence $U(G^*(\tilde{a},\tilde{b}))$ is cofinite,
and the proof of Theorem~\ref{thm:main-gen} works verbatim for
the following slight generalization.
\end{rmk*}

\begin{theorem}
\label{thm:main-gen-ab}
For a given antichain   
$G = \{\tilde{p}_1,\ldots,\tilde{p}_k\}\subseteq {\nats}^d$ and
for $\tilde{a}\leq \tilde{m}$ and $\tilde{b}\geq \tilde{M}$
where $\tilde{m}$ and $\tilde{M}$ are as in (\ref{eqn:mi-Mi}), 
we have
\[
G = S_u(D(S_d(U(G^*(\tilde{a},\tilde{b}))))).
\]
\end{theorem}
This will be used in the next Section~\ref{sec:up-down}.

Note that Theorem~\ref{thm:main-gen} states that one can retrieve
the antichain, or the generating points, $G$ from $S_d(U(G^*))$ alone. 
However, this does not mean that $S_d(U(G^*))$ can be arbitrary;
the set is derived from the (apriori unknown) set $G$. 

In the next section we use poset duality of ${\ints}^d$ to
obtain some related results from Theorem~\ref{thm:main-gen}
where we start with an arbitrary antichain $Q$ and show how it 
corresponds to the socle of a certain monomial ideal.

\section{Up-down duality}
\label{sec:up-down}

As in the derivation of Observation~\ref{obs:translation},
for any fixed point $\tilde{c}\in {\ints}^d$, the map
$\rho_{\tilde{c}} : {\ints}^d \rightarrow {\ints}^d$ given by 
$\rho_{\tilde{c}}(\tilde{x}) = \tilde{c} - \tilde{x}$ is a 
reverse-order preserving rotation, and hence an 
anti-automorphism of the poset $({\ints}^d, \leq)$. 
In particular, for $N\subseteq {\ints}^d$ we have that $N$
is an antichain iff $\rho_{\tilde{c}}(N)$ is an antichain.
Clearly $\rho_{\tilde{c}}$ is its own inverse, and so 
we have the following.
\begin{observation}
\label{obs:rotation}
For any $\tilde{c}\in {\ints}^d$ we have $\rho_{\tilde{c}}^2 = I_{{\ints}^d}$,
the identify map of ${\ints}^d$.
For any $A\subseteq {\ints}^d$ and $\tilde{c}\in {\ints}^d$ we have
\[
\rho_{\tilde{c}}(U(A)) = D(\rho_{\tilde{c}}(A)), \ \ 
\rho_{\tilde{c}}(D(A)) = U(\rho_{\tilde{c}}(A)).
\]
Further, for any upset $U\subseteq {\ints}^d$ and downset 
$D\subseteq {\ints}^d$ then $\rho_{\tilde{c}}(U)$ is a downset, 
$\rho_{\tilde{c}}(D)$ is an upset and 
\[
\rho_{\tilde{c}}(S_d(U)) = S_u(\rho_{\tilde{c}}(U)), \ \ 
\rho_{\tilde{c}}(S_u(D)) = S_d(\rho_{\tilde{c}}(D)).
\]
\end{observation}

Let $G = \{\tilde{p}_1,\ldots,\tilde{p}_k\}\subseteq {\nats}^d$
be an antichain,
and $G^* = G\cup B^*(G)$, where $B^*(G)$ is as in (\ref{eqn:B^*}).
If $\rho = \rho_{\tilde{m}+\tilde{M}}$ where $\tilde{m}$ and
$\tilde{M}$ are as in (\ref{eqn:mi-Mi}), then for any coordinate $i$
we have
\[
\min(\pi_i(\rho(G))) = m_i + M_i - (M_i-1) = m_i+1, \ \ 
\max(\pi_i(\rho(G))) = m_i + M_i - (m_i+1) = M_i-1,
\]
and so 
$G, \rho(G)\subseteq U(\tilde{m}+\tilde{1})\cap D(\tilde{M}-\tilde{1})$.

Suppose that $Q\subseteq {\nats}^d$ is such that 
$\min(\pi_i(Q)) = m_i+1$ and
$\max(\pi_i(Q)) = M_i-1$ for each $i$.
Since $\rho$ is its own inverse, there is a unique 
$G\subseteq U(\tilde{m}+\tilde{1})\cap D(\tilde{M}-\tilde{1})$ with
$Q = \rho(G)$ (namely $G = \rho(Q)$,) such that 
$\min(\pi_i(G)) = m_i+1$ and
$\max(\pi_i(G)) = M_i-1$ for each $i$, as well. Hence we have 
\begin{eqnarray*}
\rho(G^*) 
  & = & \rho(G \cup B^*(G)) \\
  & = & \rho(G)\cup\rho(B^*(G)) \\
  & = & Q \cup \{(m_1,M_2,M_3\ldots,M_d), (M_1,m_2,M_3,\ldots,M_d),
(M_1,M_2,m_3,\ldots,M_d),\\
  &   & \ldots,(M_1,M_2,M_3,\ldots,m_d)\}.
\end{eqnarray*}
Hence, if we define
\begin{eqnarray}
\label{eqn:B_*}
\lefteqn{B_*(Q) = } \nonumber \\
& & \{(m_1,M_2,M_3\ldots,M_d), (M_1,m_2,M_3,\ldots,M_d),
(M_1,M_2,m_3,\ldots,M_d), \nonumber \\
& & \ldots,(M_1,M_2,M_3,\ldots,m_d)\}, 
\end{eqnarray}
where for each $i$
$\min(\pi_i(Q)) = m_i+1$ and
$\max(\pi_i(Q)) = M_i-1$, and let $Q_* = Q \cup B_*(Q)$,
then $\rho(G^*) = Q_*$ and so by Theorem~\ref{thm:main-gen} 
and Observation~\ref{obs:rotation} we get
\begin{equation}
\label{eqn:Q}
Q = \rho(G) = \rho(S_u(D(S_d(U(G^*))))) 
  = S_d(U(S_u(D(\rho(G^*))))) = S_d(U(S_u(D(Q_*)))).
\end{equation}
Since $Q\subseteq 
U(\tilde{m}+\tilde{1})\cap D(\tilde{M}-\tilde{1})\subseteq {\nats}^d$
is a finite set, we have by (\ref{eqn:Q}) the following dual theorem of 
Theorem~\ref{thm:main-gen}.
\begin{theorem}
\label{thm:main-gen-dual}
For an antichain   
$Q = \{\tilde{p}_1,\ldots,\tilde{p}_k\}\subseteq {\nats}^d$ we
have 
\[
Q = S_d(U(S_u(D(Q_*)))),
\]
where $Q_* = Q \cup B_*(Q)$ and $B_*(Q)$ is as in (\ref{eqn:B_*}).
\end{theorem}
By the above Theorem~\ref{thm:main-gen-dual} we see that
given any antichain $Q\subseteq {\nats}^d$, then
the corresponding monomials are exactly the monomials in 
$\doc(I(Q))$, where $I(Q)$ is the monomial ideal generated
by the monomials that correspond to $S_u(D(Q_*))\subseteq {\nats}_0^d$.
This does give a positive answer to one of our motivating 
questions in Section~\ref{sec:intro}. Note, 
however, that this monomial ideal is not unique, as stated in 
the dual theorem of Theorem~\ref{thm:main-gen-ab-dual} here below.

Let $Q = \{\tilde{p}_1,\ldots,\tilde{p}_k\}\subseteq {\nats}^d$ 
be an antichain and $\tilde{a}\leq \tilde{m}$ and $\tilde{b}\geq \tilde{M}$
where $\tilde{m}$ and $\tilde{M}$ are as in (\ref{eqn:B_*}).
Further, let $\rho = \rho_{\tilde{a} + \tilde{b}}$ as defined in the beginning
of this section. If now $G = \rho(Q)$, then we have, as above,
that $\rho(G) = \rho^2(Q) = Q$ and for $G^*(\tilde{a},\tilde{b})$ as 
in the Remark right before Theorem~\ref{thm:main-gen-ab}, that
\begin{eqnarray*}
\rho(G^*(\tilde{a},\tilde{b}))
  & = & \rho(G \cup B^*(\tilde{a},\tilde{b})) \\
  & = & \rho(G)\cup\rho(B^*(\tilde{a},\tilde{b})) \\
  & = & Q \cup \{(a_1,b_2,b_3\ldots,b_d), (b_1,a_2,b_3,\ldots,b_d),
(b_1,b_2,a_3,\ldots,b_d),\ldots,(b_1,b_2,b_3,\ldots,a_d)\}.
\end{eqnarray*}
So, again, we can define
\begin{equation}
\label{eqn:B_*(a,b)}
B_*(\tilde{a},\tilde{b}) = \{(a_1,b_2,b_3\ldots,b_d), (b_1,a_2,b_3,\ldots,b_d),
(b_1,b_2,a_3,\ldots,b_d),\ldots,(b_1,b_2,b_3,\ldots,a_d)\}.
\end{equation}
If now 
\[
Q_*(\tilde{a},\tilde{b}) := Q \cup B_*(\tilde{a},\tilde{b}),
\]
then $\rho(G^*(\tilde{a},\tilde{b})) = Q_*(\tilde{a},\tilde{b})$, and 
so by Theorem~\ref{thm:main-gen-ab} and Observation~\ref{obs:rotation} we get
\[
Q = \rho(G) = \rho(S_u(D(S_d(U(G^*(\tilde{a},\tilde{b}))))))
  = S_d(U(S_u(D(\rho(G^*(\tilde{a},\tilde{b})))))) 
  = S_d(U(S_u(D(Q_*(\tilde{a},\tilde{b}))))).
\]
Since this holds for any $\tilde{a}\leq \tilde{m}$ and 
$\tilde{b}\geq \tilde{M}$, we have a following dual theorem
of Theorem~\ref{thm:main-gen-ab}.
\begin{theorem}
\label{thm:main-gen-ab-dual}
For an antichain 
$Q = \{\tilde{p}_1,\ldots,\tilde{p}_k\}\subseteq {\nats}^d$ and
for $\tilde{a}\leq \tilde{m}$ and $\tilde{b}\geq \tilde{M}$
where $\tilde{m}$ and $\tilde{M}$ are as in (\ref{eqn:B_*}),
we have
\[
Q = S_d(U(S_u(D(Q_*(\tilde{a},\tilde{b}))))).
\]
where $Q_*(\tilde{a},\tilde{b}) = Q \cup B_*(\tilde{a},\tilde{b})$ 
and $B_*(\tilde{a},\tilde{b})$ is as in (\ref{eqn:B_*(a,b)}).
\end{theorem}
\begin{rmks*}\ 
\begin{enumerate}[label=(\roman*)]
\item  Note that Theorem~\ref{thm:main-gen-ab-dual} 
is valid for {\em any} $\tilde{a},\tilde{b}\in {\ints}^d$ 
that satisfy $\tilde{a}\leq \tilde{m}$ and $\tilde{b}\geq \tilde{M}$.
This will be used in the next section.
\item Note that if $Q_*(\tilde{a},\tilde{b})$ is given
(without the prior knowledge of $B_*(\tilde{a},\tilde{b})$), then
we can retrieve $B_*(\tilde{a},\tilde{b})$ as in 
in (\ref{eqn:B_*(a,b)}) from $Q_*(\tilde{a},\tilde{b})$,
and hence the set $Q$, if
$\tilde{a}\leq \tilde{m}$ and $\tilde{b}\geq \tilde{M}$. 
\item Last but not least, note that by Observations~\ref{obs:translation}
and~\ref{obs:rotation}, all Theorems~\ref{thm:main-gen}, \ref{thm:main-gen-ab}, 
\ref{thm:main-gen-dual} and~\ref{thm:main-gen-ab-dual} hold for any
antichain $Q\subseteq {\ints}^d$, and not merely those
of ${\nats}^d$.
\end{enumerate}
\end{rmks*}

\begin{example*}
Consider the case $d=3$ and the set
$Q = \{ (2,2,3), (3,3,2)\}\subseteq {\nats}^3$.
We use the above Theorem~\ref{thm:main-gen-dual} to obtain
an upset $U(S_d(D(Q_*)))$, whose minimal points 
$S_u(D(Q_*))\subseteq {\nats}_0^3$ correspond to the
generators of a monomial ideal $I$ where $\overline{\soc(I)}$
is spanned by the image of $M = \{x^2y^2z^3,x^3y^3z^2\}$ in $R/I$, 
i.e.~the monomials corresponding to the set $Q$, as follows.

By (\ref{eqn:B_*}) we have here that
$B_*(Q) = \{(4,4,1), (4,1,4), (1,4,4)\}$ and hence
\[
Q_* = Q\cup B_*(Q) = \{(2,2,3), (3,3,2), (4,4,1), (4,1,4), (1,4,4)\},
\]
and so $S_u(D(Q_*)) = \{(2,2,4), (2,3,3), (2,4,2), (3,2,3), (4,2,2)\}$.
By Theorem~\ref{thm:main-gen-dual} we now have
\begin{eqnarray*}
\lefteqn{S_d(U(S_u(D(Q_*))))}  \\ 
& = & S_d(U(S_u(D((2,2,3), (3,3,2), (4,4,1), (4,1,4), (1,4,4))))) \\
& = & S_d(U((2,2,4), (2,3,3), (2,4,2), (3,2,3), (4,2,2))) \\
& = & \{ (2,2,3), (3,3,2)\} \\
& = & Q.
\end{eqnarray*}
Therefore, the monomial ideal 
$I_1 = (x^2y^2z^4, x^2y^3z^3, x^2y^4z^2, x^3y^2z^3, x^4y^2z^2)\subseteq K[x,y,z]$
has $\overline{\soc(I_1)}$ spanned by the image of
$M = \{x^2y^2z^3,x^3y^3z^2\}$ in $R/{I_1}$ as a $k$-vector space.

Since $\tilde{a} = (0,0,1)\leq (1,1,1) =\tilde{m}$ and
$\tilde{b} = (5,6,7)\geq (4,4,4) = \tilde{M}$ in ${\ints}^3$, 
we have for 
\[
Q_*(\tilde{a},\tilde{b}) = Q_*((0,0,1),(5,6,7)) =
\{(2,2,3), (3,3,2), (5,6,1), (5,0,7), (0,6,7)\},
\]
by Theorem~\ref{thm:main-gen-ab-dual} that
\begin{eqnarray*}
\lefteqn{S_d(U(S_u(D(Q_*((0,0,1),(5,6,7))))))}  \\ 
& = & S_d(U(S_u(D((2,2,3), (3,3,2), (5,6,1), (5,0,7), (0,6,7))))) \\
& = & S_d(U((1,1,4), (1,3,3), (1,4,2), (3,1,3), (4,1,2)))\\
& = & \{ (2,2,3), (3,3,2)\} \\
& = & Q.
\end{eqnarray*}
as well. Hence, the monomial ideal 
$I_2 = (xyz^4, x^1y^3z^3, xy^4z^2, x^3yz^3, x^4yz^2)\subseteq K[x,y,z]$
also has $\overline{\soc(I_2)}$ spanned by the image of
$M = \{x^2y^2z^3,x^3y^3z^2\}$ in $R/{I_2}$ as a $k$-vector space.
\end{example*}

\begin{rmks*}\ 
\begin{enumerate}[label=(\roman*)]
\item As with many formulae, when it comes utilizing 
them to compute specific values, the compact forms and shortness
is not always a guarantee for a fast evaluation. Conversely, 
a seemingly ugly expression can many times be much better to
use to obtain specific values in a fast and an efficient manner.
The computation of $S_u(D(Q_*))$, from an antichain $Q$ consisting 
of $k$ points from ${\nats}^d$ as in Theorem~\ref{thm:main-gen-dual},
can for each fixed $k$ be done in polynomial time
in $k$ alone. In fact, it can be done in $O(k^d)$-time,
although the exact or a tight upper bound of its complexity
is hard to come by. 
\item We have so far assumed $G,Q\subseteq {\nats}^d$,
only to guarantee that $G^*, Q_*\subseteq {\nats}_0^d$. However,
general $G,Q\subseteq {\nats}_0^d$ will imply 
$G^*, Q_*\subseteq\{-1,0,1,2,\ldots\}^d\subseteq {\ints}^d$ 
which does not correspond
to a set of monomials from $[x_1,\ldots,x_d]$ but rather 
monomials from $[x_1,x_1^{-1},\ldots,x_d,x_d^{-1}]$.
\end{enumerate}
\end{rmks*}

\section{Zero-dimensional monomial ideals}
\label{sec:zero-monomial}

For a field $K$, the ring $R = K[x_1,\ldots,x_d]$,
the maximal ideal $\m$ of $R$ and a monomial ideal $I$ of
$R$, the motivation for this section is the following question.
\begin{question}
\label{qst:retrieve}
Under what circumstances can we retrieve a unique monomial ideal $I$
from the monomials of $\soc(I) = \soc_{\m}(I)$ that are not in $I$?
\end{question}
We saw in the last example in the previous Section~\ref{sec:up-down}
that both the following ideals of $R = K[x,y,z]$
\begin{eqnarray*}
I_1 & = & (x^2y^2z^4, x^2y^3z^3, x^2y^4z^2, x^3y^2z^3, x^4y^2z^2) \mbox{ and } \\
I_2 & = & (xyz^4, x^1y^3z^3, xy^4z^2, x^3yz^3, x^4yz^2)
\end{eqnarray*}
have $\overline{\soc(I_1)}$ and resp.~$\overline{\soc(I_2)}$
spanned by the image of $M = \{x^2y^2z^3,x^3y^3z^2\}$
in $R/{I_1}$ and resp.~$R/{I_2}$ as a $k$-vector space. So
Question~\ref{qst:retrieve} has in general a negative answer.
In fact, for any monomial ideal $I$ it is easy to construct a rich
family of monomial ideals such that $\doc(I') = \doc(I)$ 
for any ideal $I'$ in that family.

If, however, we assume $\dim(I) = 0$ then, we will see shortly,
Question~\ref{qst:retrieve} has a positive answer.
Zero-dimensional monomial ideals of $R = K[x_1,\ldots,x_d]$ 
constitute an interesting yet fairly general class of monomial ideals 
for numerous reasons: the quotient ring $R/I$ is a local ring with
a unique prime and maximal ideal, $R/I$ is finite dimensional over
$k$, and their variety consist of a single point $\tilde{0}$, to name
a few. 

For a zero-dimensional ideal $I\neq\m$ of $R$ we note
that the image of $\soc_{\m}(I)$ in $R/I$ corresponds to an antichain
$Q$ of ${\nats}_0^d$. As mentioned in the remark after 
Theorem~\ref{thm:main-gen-ab-dual}, we then have for our 
$Q\subseteq {\nats}_0^d$ that 
$Q = S_d(U(S_u(D(Q_*(\tilde{a},\tilde{b})))))$ for any 
suitable $\tilde{a}, \tilde{b}\in {\ints}^d$. In particular, since we have 
$-\tilde{1} = (-1,\ldots,-1)\leq \tilde{m}$ from  
Theorem~\ref{thm:main-gen-ab-dual}, then for any $\tilde{b}\geq\tilde{M}$
we have $Q = S_d(U(S_u(D(Q_*(-\tilde{1},\tilde{b})))))$. We now
briefly argue that the upset 
$U(S_u(D(Q_*(-\tilde{1},\tilde{b}))))$ corresponds to a monomial ideal 
$I$ of $R$ of dimension zero: first we note that by definition 
(\ref{eqn:B_*(a,b)}) we have 
\begin{eqnarray*}
D(Q_*(-\tilde{1},\tilde{b})) & \supseteq & D(B_*(-\tilde{1},\tilde{b})) \\
 & = & D((-1,b_2,b_3\ldots,b_d), (b_1,-1,b_3,\ldots,b_d), 
\ldots,(b_1,b_2,b_3,\ldots,-1) 
\end{eqnarray*}
and hence 
$S_u(D(Q_*(-\tilde{1},\tilde{b})))\subseteq U(\tilde{0}) = {\nats}_0^d$,
and so the generators for the upset 
$U(S_u(D(Q_*(-\tilde{1},\tilde{b}))))$ correspond to monomials of 
$R = K[x_1,\ldots,x_d]$. Secondly, since $\tilde{b}\geq\tilde{M}$
as in Theorem~\ref{thm:main-gen-ab-dual}, then for each $i\in [d]$
we have $b_i\tilde{e}_i = 
(0,\ldots,0,b_i,0,\ldots,0)\not\in D(Q_*(-\tilde{1},\tilde{b}))$
and $-\tilde{e}_i = (0,\ldots,0,-1,0,\ldots,0)\in D(Q_*(-\tilde{1},\tilde{b}))$,
and so there is a smallest $c_i\in{\nats}_0$, $c_i\leq b_i$ with 
$c_i\tilde{e}_i\not\in D(Q_*(-\tilde{1},\tilde{b}))$. Since 
$c_i\tilde{e}_i -\tilde{e}_j\in D(Q_*(-\tilde{1},\tilde{b}))$ for
each $j\neq i$ and 
$c_i\tilde{e}_i - \tilde{e}_i = (c_i-1)\tilde{e}_i
\in D(Q_*(-\tilde{1},\tilde{b}))$, we have that 
$c_i\tilde{e}_i\in S_u(D(Q_*(-\tilde{1},\tilde{b})))$ for each $i$, 
which means that $x_i^{c_i}$ is an element of the monomial ideal of $I$ that
corresponds to $U(S_u(D(Q_*(-\tilde{1},\tilde{b}))))$, showing that
$I$ is indeed of dimension zero.

To see that the zero-dimensional ideal $I$ is unique, it suffices
to show that the upset from above 
$U = U(S_u(D(Q_*(-\tilde{1},\tilde{b}))))$ is the unique
upset with $Q = S_d(U)$. Assume $I_1$ and $I_2$ are two zero-dimensional
monomial ideals with the same socle w.r.t.~the maximal ideal $\m$ of $R$. 
As each $I_i$ is zero-dimensional, it corresponds to an upset 
$U(G_i^*(\tilde{0},\tilde{b}_i))$ where 
$G_i^*(\tilde{0},\tilde{b}_i) = G_i\cup B^*(\tilde{0},\tilde{b}_i)$
is an antichain of ${\nats}_0^d$ as in (\ref{eqn:B^*(a,b)}). 
By our assumption we have 
$S_d(U(G_1^*(\tilde{0},\tilde{b}_1))) = Q 
= S_d(U(G_2^*(\tilde{0},\tilde{b}_2)))$ and hence by 
Theorem~\ref{thm:main-gen-ab} we then get
\[
G_1 
= S_u(D(S_d(U(G_1^*(\tilde{0},\tilde{b}_1)))))
= S_u(D(Q))
= S_u(D(S_d(U(G_1^*(\tilde{0},\tilde{b}_1)))))
= G_2.
\]
Also, directly from $S_d(U(G_1^*(\tilde{0},\tilde{b}_1))) = Q 
= S_d(U(G_2^*(\tilde{0},\tilde{b}_2)))$, we obtain for each 
$i\in[d]$ that 
\[
b_{1\/i} = \pi_i(\tilde{b}_1) = \max(\{\pi_i(\tilde{q}) : \tilde{q}\in Q\})
= \pi_i(\tilde{b}_2) = b_{1\/2},
\]
and hence $\tilde{b}_1 = \tilde{b}_2$. Therefore 
\[
G_1^*(\tilde{0},\tilde{b}_1) = G_1\cup B^*(\tilde{0},\tilde{b}_1)
= G_2\cup B^*(\tilde{0},\tilde{b}_2) = G_2^*(\tilde{0},\tilde{b}_2),
\]
and so 
$U(G_1^*(\tilde{0},\tilde{b}_1) = U(G_2^*(\tilde{0},\tilde{b}_2)$.
This means that the corresponding zero-dimensional monomial
ideals are equal $I_1 = I_2$. This yields a positive answer
to Question~\ref{qst:retrieve} for zero-dimensional monomial ideals,
as stated in the following proposition, in which the set $S$ of
monomials corresponds to our antichain $Q\subseteq {\nats}_0^d$ 
from above.
\begin{proposition}
\label{prp:retrieve}
For any 
non-empty set $S$ of non-comparable monomials  of
$R = K[x_1,\ldots,x_d]$ there is a unique zero-dimensional ideal
$I$ of $R$ with $\doc(I) = S$.
\end{proposition}
For the remainder of this section we discuss in further
detail the cases when the set $S$ of monomials in 
Proposition~\ref{prp:retrieve}, that corresponds to an 
antichain $Q\subseteq {\nats}_0^d$, has cardinality $1$ and $2$,
that is $|S| = |Q| \in \{1,2\}$.

\subsection*{The case when $Q$ is a singleton}
An interesting case of Proposition~\ref{prp:retrieve} this is when $|Q|=1$, 
so the antichain $Q$ of ${\nats}_0^d$
consists of just one point, say $Q = \{\tilde{p}\}$.
As before, the unique zero-dimensional monomial ideal 
corresponds to the upset $U = U(S_u(D(Q_*(-\tilde{1},\tilde{b}))))$ 
where $Q = S_d(U)$. If now $\tilde{p} = (p_1,\ldots,p_d)\in {\nats}_0^d$,
we note that the downset $D(Q_*(-\tilde{1},\tilde{b}))$ is generated
by an antichain consisting of $d+1$ elements
\[
Q_*(-\tilde{1},\tilde{b}) = \{\tilde{p}, (-1,b_2,b_3\ldots,b_d), 
(b_1,-1,b_3,\ldots,b_d), \ldots,(b_1,b_2,b_3,\ldots,-1)\}.
\]
Since each element $\tilde{r}\in S_u(D(Q_*(-\tilde{1},\tilde{b})))$ is,
by definition of $S_u(D)$, uniquely determined by $d$ distinct
elements of $Q_*(-\tilde{1},\tilde{b})$ (by the conditions
(i) $\tilde{r}\not\in Q_*(-\tilde{1},\tilde{b})$ and (ii) 
$\tilde{r}-\tilde{e}_i\in Q_*(-\tilde{1},\tilde{b})$ for each 
$i\in[d]$), we have that $S_u(D(Q_*(-\tilde{1},\tilde{b})))$
is a subset of the possible $\binom{d+1}{d} = d+1$ elements 
$\{\tilde{0},(p_1+1)\tilde{e}_1,\ldots,(p_d+1)\tilde{e}_d\}$,
and hence, as it is an antichain, we obtain
\[
S_u(D(Q_*(-\tilde{1},\tilde{b}))) = 
\{(p_1+1)\tilde{e}_1,\ldots,(p_d+1)\tilde{e}_d\}\},
\]
which means that the upset $U = U(S_u(D(Q_*(-\tilde{1},\tilde{b}))))$ 
corresponds to a monomial ideal $I$ of the form
$I = (x_1^{p_1+1},\ldots,x_d^{p_d+1})$.

That $|S| = 1$ means that $\doc(I) = S$ 
contains the unique generator of the ideal $\overline{\soc_{\m}(I)}$
of $R/I$. This means that $R/I$ is a zero dimensional local Gorenstein 
ring~\cite[Prop.~21.5]{Eisenbud}. Clearly, if our monomial ideal
$I$ is generated by pure powers of the variables $x_1,\ldots,x_d$,
then $R/I$ is a zero dimensional local Gorenstein ring with
a unique monomial in $\doc(I)$. Hence, we obtain as a corollary
the following description of Artinian Gorenstein rings that 
are defined by monomials~\cite{Beintema}, see also~\cite{Wol16}. 
\begin{corollary}
\label{cor:mon-Gorenstein}
Let $R = K[x_1,\ldots,x_d]$ and $I$ be a monomial ideal of $R$. 
Then $R/I$ is a zero dimensional local Gorenstein ring
if, and only if, $I = (x_1^{p_1+1},\ldots,x_d^{p_d+1})$ for
some $\tilde{p}\in {\nats}_0^d$, in which case 
$\doc(I) = \{\tilde{x}^{\tilde{p}} = x_1^{p_1}\cdots x_d^{p_d}\}$.
\end{corollary}
For a general $Q\subseteq {\ints}^d$ we have, as mentioned
here above, that every element of $S_u(D(Q))$ is uniquely determined by
$d$ distinct elements of $Q$. More specifically we have the following.
\begin{claim}
\label{clm:d-elts}
For an antichain $G\subseteq {\ints}^d$ and $\tilde{r}\in S_d(U(G))$
there are distinct $\tilde{p}_1,\ldots,\tilde{p}_d\in G$ such
that (i) $\tilde{r}\in S_d(U(\tilde{p}_1,\ldots,\tilde{p}_d))$
and (ii) $\tilde{r} = (p_{1\/1}+1,\ldots,p_{d\/d}+1)$.

Dually, for an antichain $Q\subseteq {\ints}^d$ and $\tilde{r}\in S_u(D(Q))$
there are distinct $\tilde{p}_1,\ldots,\tilde{p}_d\in Q$ such
that (i) $\tilde{r}\in S_d(U(\tilde{p}_1,\ldots,\tilde{p}_d))$
and (ii) $\tilde{r} = (p_{1\/1}-1,\ldots,p_{d\/d}-1)$.
\end{claim}
We can in similar fashion, as for Corollary~\ref{cor:mon-Gorenstein},
consider the case when $|S|=k$, that is
the corresponding antichain $Q$ of ${\ints}_0^d$ contains
$k$ points $Q = \{\tilde{p}_1,\ldots,\tilde{p}_k\}$, where
each $\tilde{p}_i = (p_{i\/1},\ldots,p_{i\/d})$. In this case we have
$Q_*(-\tilde{1},\tilde{b}) = Q\cup B_*(-\tilde{1},\tilde{b})$
where $B_*(-\tilde{1},\tilde{b})$ is as in (\ref{eqn:B_*(a,b)}),
and hence $b_i \geq \max(p_{i\/1}+1,\ldots,p_{i\/d}+1)$ for each index $i$, 
and so the downset $D(Q_*(-\tilde{1},\tilde{b}))$ is generated
by an antichain of $d+k$ elements. For convenience 
we let 
\[
\tilde{b}(i):= (b_1,\ldots,b_{i-1},-1,b_{i+1},\ldots,b_d),
\]
for each $i\in [d]$, so 
$B_*(-\tilde{1},\tilde{b}) = \{\tilde{b}(1),\ldots,\tilde{b}(d)\}$.
The following lemma provides a useful tool.
\begin{lemma}
\label{lmm:p1-pk-box}
For ${\cal{I}} = \{i_1,\ldots i_k\}\subseteq [d]$ and 
$D = D(\tilde{p}_1,\ldots,\tilde{p}_k,\tilde{b}(1),\ldots,
\widehat{\tilde{b}(i_1)},\ldots,\widehat{\tilde{b}(i_k)},\ldots,\tilde{b}(d))$
we have
\[
S_u(D) =
\left\{
\begin{array}{ll}
  \left\{\sum_{\ell = 1}^k r_{i_{\ell}}\tilde{e}_{i_{\ell}}\right\} &
  \mbox{ if } S_u(D(\pi_I(\tilde{p}_1),\ldots,\pi_I(\tilde{p}_k)))
  = \{(r_{i_1},\ldots,r_{i_k})\} \neq\emptyset \\
    \emptyset & \mbox{ otherwise. }
\end{array}
\right.
\]
\end{lemma}
\begin{proof}
For $\tilde{r}\in S_u(D)$ we have by definition that $\tilde{r}\not\in D$
and $\tilde{r} - \tilde{e}_{\ell} \in D$ for each $\ell$.
Since $D\subseteq D(\tilde{b})$ we have then $r_k\leq b_k$
for each $k\neq \ell$.
As this holds for each $\ell$ we have $\tilde{r}\leq\tilde{b}$. 

Further, for $\tilde{r}\in S_u(D)$, there is, analogous to 
Claim~\ref{clm:d-elts}, a permutation $\sigma\in S_d$ such
that $\tilde{r} - \tilde{e}_{\sigma(\ell)} \leq \tilde{b}(\ell)$
for each index $\ell\not\in {\cal{I}}$ and
$\tilde{r} - \tilde{e}_{\sigma(i_{\ell})}\leq\tilde{p}_{\ell}$
for $1\leq \ell\leq k$.
Since $\tilde{r}\not\in D$ we hence have for each $\ell\not\in {\cal{I}}$ that 
\begin{equation}
\label{eqn:b+1}
r_{\sigma(\ell)} =
\left\{
\begin{array}{ll}
  b_{\sigma(\ell)} + 1 & \mbox{ if } \sigma(\ell)\neq \ell \\
  0 & \mbox{ if } \sigma(\ell) = \ell
\end{array}
\right.
\end{equation}
and $r_{\sigma(\ell)} = p_{\ell\/\sigma(\ell)} + 1$ for each $\ell\in {\cal{I}}$.
Since $\tilde{r}\leq\tilde{b}$ we must by (\ref{eqn:b+1}) have that
$r_{\sigma(\ell)} = 0$ for each $\ell\not\in {\cal{I}}$ and therefore
$\sigma(\ell) = \ell$
for each $\ell\not\in {\cal{I}}$. Hence we have $r_{\ell} = 0$ for each
$\ell\not\in {\cal{I}}$.
Consequently $\sigma \in S({\cal{I}})$ is a permutation only on ${\cal{I}}$
and leaves every other element of $[d]\setminus {\cal{I}}$ fixed.
In particular we have that
\[
\tilde{r} = \sum_{\ell = 1}^k r_{i_{\ell}}\tilde{e}_{i_{\ell}} =
\sum_{i\in {\cal{I}}}r_i\tilde{e}_i.
\]
Since $\tilde{r}\not\in D$ we have
$\tilde{r}\not\in D(\tilde{p}_{i_1},\ldots,\tilde{p}_{i_k})$ and
since $\tilde{r} - \tilde{e}_{\sigma(i_{\ell)}}\leq\tilde{p}_{\ell}$
for each $\ell\in [k]$, or
$\tilde{r} - \tilde{e}_{\ell} \leq \tilde{p}_{\gamma(\ell)}$
for each $\ell\in {\cal{I}}$, where $\gamma : {\cal{I}} \rightarrow [k]$
is the map $i_{\ell} \mapsto \ell$, we have
$(r_{i_1},\ldots,r_{i_k})
= \pi_{\cal{I}}(\tilde{r})\in S_u(D(\pi_{\cal{I}}(\tilde{p}_1),
\ldots,\pi_{\cal{I}}(\tilde{p}_k)))$,
which by Claim~\ref{clm:d-elts} is uniquely determined.
\end{proof}
From the above proof we note that if $S_u(D)\neq\emptyset$ and
$\tilde{r}\in S_u(D)$, then $r_{\sigma(i_{\ell})} = P_{\ell\/\sigma(i_{\ell})}+1$
for $1\leq \ell\leq k$ and so
$r_{\sigma(i_{\ell})} = P_{\gamma(i_{\ell})\/\sigma(i_{\ell})}+1$ 
for $1\leq \ell\leq k$, or $r_{\ell} = p_{\beta(\ell)\/\ell}+1$
for each $\ell\in {\cal{I}}$, where $\beta = \gamma\circ{\sigma}^{-1}$,
and so
\[
\tilde{r} = (p_{\beta(i_i)\/i_1}+1)\tilde{e}_{i_1}
+ \cdots + (p_{\beta(i_k)\/i_k}+1)\tilde{e}_{i_k}
\]
for the bijection $\beta : {\cal{I}}\rightarrow [k]$.
From this we get the following.
\begin{observation}
\label{obs:p1-pk-box}
Let ${\cal{I}} = \{i_1,\ldots i_k\}\subseteq [d]$ and 
$D = D(\tilde{p}_1,\ldots,\tilde{p}_k,\tilde{b}(1),\ldots,
\widehat{\tilde{b}(i_1)},\ldots,\widehat{\tilde{b}(i_k)},\ldots,\tilde{b}(d))$.
If $S_u(D)\neq\emptyset$ and $\tilde{r}\in S_u(D)$ is its unique element,
then there is a permutation $\alpha\in S_k$ such that
\[
\pi_{\cal{I}}(\tilde{r}) = (p_{\alpha(1)\/i_1}+1,\ldots,p_{\alpha(k)\/i_k}+1)\in
S_u(D(\pi_{\cal{I}}(\tilde{p}_1),\ldots,\pi_{\cal{I}}(\tilde{p}_k))).
\]
\end{observation}

\subsection*{The case when $k=|Q|=2$} 
The second special case of $k=2$, so $Q = \{\tilde{p},\tilde{q}\}$,
is, as we will see here below, simple enough so that we can
state some conclusion in a direct and uncluttered manner.
The following follows directly from the above
Lemma~\ref{lmm:p1-pk-box} and Observation~\ref{obs:p1-pk-box}.
\begin{corollary}
\label{cor:pq-box}
For two distinct indices $i,j\in [d]$ and 
$D = D(\tilde{p},\tilde{q},\tilde{b}(1),\ldots,
\widehat{\tilde{b}(i)},\ldots,\widehat{\tilde{b}(j)},\ldots,\tilde{b}(d))$
we have
\[
S_u(D) = 
\left\{
\begin{array}{ll}
\{(p_i+1)\tilde{e}_i + (q_j+1)\tilde{e}_j\} & 
  \mbox{ if } p_i<q_i\mbox{ and }q_j < p_j \\
\{(q_i+1)\tilde{e}_i + (p_j+1)\tilde{e}_j\} & 
  \mbox{ if } q_i<p_i\mbox{ and }p_j < q_j \\
\emptyset & \mbox{ otherwise.}
\end{array}
\right.
\]
\end{corollary}
For the antichain $Q = \{\tilde{p},\tilde{q}\}$ of ${\nats}_0^d$ 
we have a {\em pseudo-partition} $[d] = A\cup B\cup C$
(some parts could be empty) 
where $p_i<q_i$ for all $i\in A$, $p_i=q_i$ for all $i\in B$,
and $p_i>q_i$ for all $i\in C$. We have here that 
\[
Q_*(-\tilde{1},\tilde{b}) = Q\cup B_*(-\tilde{1},\tilde{b})
= \{\tilde{p},\tilde{q}, \tilde{b}(1),\ldots,\tilde{b}(d)\}.
\]
Since $S_u(D)$ forms an antichain, for any downset $D\subseteq {\ints}^d$,
and each element of $S_u(Q_*(-\tilde{1},\tilde{b})$ Is uniquely 
determined by a $d$-subset of $Q_*(-\tilde{1},\tilde{b})$, then
$S_u(Q_*(-\tilde{1},\tilde{b})$ is among the at most 
$\binom{d+2}{d} = \binom{d+2}{2}$ elements of
$\bigcup_{Q'}S_u(D(Q'))$, where the union is taken over all
$Q'\subseteq Q_*(-\tilde{1},\tilde{b})$ with cardinality $d$.
Since $[d] = A\cup B\cup C$ is a pseudo-partition, then
by letting $|A| = {\mathbf{a}}, |B| = {\mathbf{b}}$ and 
$|C| = {\mathbf{c}}$ we have
$|A| + |B\cup C| = {\mathbf{a}} + {\mathbf{b}} + {\mathbf{c}} = d$, 
and therefore, by Corollary~\ref{cor:pq-box}, we then have the following.
\begin{proposition}
\label{prp:type 2}
For an antichain $Q = \{\tilde{p},\tilde{q}\}$ of of ${\ints}_0^d$
and the corresponding pseudo-partition $[d] = A\cup B\cup C$
we have 
\begin{eqnarray*}
\lefteqn{S_u(D(Q_*(-\tilde{1},\tilde{b}))) = } \\ 
& & \{(p_i + 1)\tilde{e}_i + (q_j + 1)\tilde{e}_j : (i,j)\in A\times C\}
\cup 
\{(p_i + 1)\tilde{e}_i : i\in B\cup C\}
\cup 
\{(q_i + 1)\tilde{e}_i : i\in A\},
\end{eqnarray*}
and hence $|S_u(D(Q_*(-\tilde{1},\tilde{b})))| = {\mathbf{a}}{\mathbf{c}} + d$.
\end{proposition}
Note that 
\begin{eqnarray*}
&   & \{(p_i + 1)\tilde{e}_i : i\in B\cup C\} \cup 
\{(q_i + 1)\tilde{e}_i : i\in A\} \\
& = & \{(p_i + 1)\tilde{e}_i : i\in B\cup C\} \cup 
\{(q_i + 1)\tilde{e}_i : i\in A\cup B\} \\
& = & \{(p_i + 1)\tilde{e}_i : i\in C\} \cup 
\{(q_i + 1)\tilde{e}_i : i\in A\cup B\}.
\end{eqnarray*}
In the above display, the first and the last unions are
pseudo-partitions and not symmetric,  whereas the middle union 
is symmetric but not disjoint.

Similar to Corollary~\ref{cor:mon-Gorenstein} describing
Artinian Gorenstein rings defined by monomial from~\cite{Beintema},
we can interpret the above Proposition~\ref{prp:type 2} for
Artinian rings defined by monomials that are
{\em almost Gorenstein} as defined in~\cite{Huneke-Vraciu}
as those monomial ideals $I$ of $R = K[x_1,\ldots,x_d]$ with
$\doc(I) = S$ from Proposition~\ref{prp:retrieve}
being a set of ``small'' cardinality. More specifically,
if $I$ is a zero dimensional monomial ideal of $R$ with $|S| = k\geq 1$,
then $I$ is said to be of {\em type $k$}, and so type 1 zero
dimensional ideals are exactly the Artinian Gorenstein monomial
ideals~\cite{Huneke-Vraciu}. The above Proposition~\ref{prp:type 2}
gives a complete description of the type 2 zero dimensional monomial
ideals of $R$.
\begin{corollary}
\label{cor:mon-type 2}
Let $R = K[x_1,\ldots,x_d]$ and $I$ be a monomial ideal of $R$. 
Then $R/I$ is a zero dimensional local type 2 ring
if, and only if, for two incomparable $\tilde{p},\tilde{q}\in {\ints}_0^d$
with corresponding pseudo-partition $[d] = A\cup B\cup C$ we have
\[
I = (x_i^{p_i+1}x_j^{q_j+1} , x_h^{p_h+1} , x_k^{q_k+1}
    : (i,j)\in A\times C, h\in B\cup C, k\in A),
\]
in which case $\doc(I) = S = \{\tilde{x}^{\tilde{p}}, \tilde{x}^{\tilde{q}}\}$.
\end{corollary}

\begin{rmks*}\ 
\begin{enumerate}[label=(\roman*)]
\item Note that Corollary~\ref{cor:mon-type 2} characterizes
completely those monomial ideals $I$ of $R = K[x_1,\ldots,x_d]$
that make $R/I$ a local Artinian ring of type 2, in a purely
combinatorial way, using only the poset structure of ${\ints}^d$,
\item Comparing with \cite[Example 4.3]{Huneke-Vraciu} we
see that the analysis of type 2 monomial ideal done there 
is slightly different from ours from Corollary~\ref{cor:mon-type 2}, in that 
there in~\cite{Huneke-Vraciu} the authors describe when exactly the condition
$J_1:J_2 + J_2:J_1 \supseteq \m = (x_1,\ldots,x_d)$ from~\cite[Theorem 4.2]{Huneke-Vraciu} holds, where $I = J_1\cap J_2$ 
is the irredundant intersection of two irreducible ideals $J_1$ and $J_2$, 
each of which must be monomial ideals generated by powers of the
variables $x_1,\ldots,x_d$.
\end{enumerate}
\end{rmks*}

\section{Almost Gorenstein monomial ideals with $k\geq 3$}
\label{sec:kgeq3}

\subsection*{The general case for $k=|Q|\geq 3$} 
We have so far analysed fully the structure of the antichain 
$S_u(Q_*(-\tilde{1},\tilde{b}))$ of ${\ints}_0^d$ in the cases when 
$k=|Q|\in\{1,2\}$, the antichain corresponding to the generators
of a zero-dimensional monomial ideal with $\doc(I)$ corresponding to
$Q$. These two cases are unique in that the form of the antichain $Q$ is
quite simple: an antichain $Q=\{\tilde{p},\tilde{q}\}$ in ${\ints}^d$ 
with the corresponding pseudo-partition $[d] = A\cup B\cup C$ as indicated
in Proposition~\ref{prp:type 2} must have both parts $A$ and $C$ non-empty
and so ${\mathbf{a}}{\mathbf{c}}\geq 1$.
This means that the parts (here $A$ and $C$) corresponding to the two
possible strict orderings of the coordinates, namely $p_i<q_i$ or $p_i>q_i$,
are each nonempty. Needless to say, for an antichain $Q\subseteq {\ints}^d$
where $k=|Q|\geq 3$, this need not be case. The possibilities are simply
to great. There are plenty antichains (in fact, vastly most of them, 
as we shall see momentarily) where not all strict orderings of
the coordinates are present.

Analogous to the pseudo-partition $[d] = A\cup B\cup C$ 
from Proposition~\ref{prp:type 2} we have for general $k\in\nats$ 
a pseudo-partition
\begin{equation}
\label{eqn:partk}
[d] = \bigcup_{\omega\in W(k)}A_{\omega},
\end{equation}
where $W(k)$ is the collection of {\em weak orderings} of the
set $[k]$: the collection of distinct orderings of the
elements where equality is allowed. In particular we clearly have
for each $k\geq 2$ that
$|W(k)| > k!$, the number of strict orderings of $[k]$, and so
the union in (\ref{eqn:partk}) is large.
More specifically, for $k=1,2$ and $3$ there
are $1, 3$ and $13$ weak orderings in $W(k)$ respectively, and, in general,
$a(k) = |W(k)|$ is the $k$-th {\em ordered Bell number} or
{\em Fubini number} given explicitly by
\[
a(k) = \sum_{i = 1}^k\stirling{k}{i}i!,
\]
where $\stirling{k}{i} = S(k,i)$ is the 
Stirling number of the second kind~\cite{Stirling-2} and where the 
exponential generating function for the corresponding sequence 
$(a(k))_{k\geq 0} = (1,1,3,13,75,541,\ldots)$ is given by  
\[
\sum_{k=0}^{\infty}\frac{a(k)}{k!}x^k = \frac{1}{2-e^x}.
\]
Asymptotically
$a(k)\approx \frac{n!}{2(\log2)^{n+1}} = \frac{1}{2}(1.4427\cdots)^{n+1}n!$,
where $\log$ is the natural logarithm and so $\log2 < 1$, and so we see that
the number $k!$ of strict orderings of $[k]$ constitutes only
a small fraction of the $a(k)$ weak orderings of $[k]$. The ordered Bell
numbers have been studied extensively~\cite{OBn-wiki, A000670}.
\begin{definition}
\label{def:order-generic}
An antichain $Q = \{\tilde{p}_1,\ldots,\tilde{p}_k\}\subseteq {\ints}^d$
is called {\em order-generic} if in the corresponding pseudo-partition 
(\ref{eqn:partk}) for each $\omega\in W(k)$ of strict orderings of $[k]$ 
we have $A_{\omega}\neq\emptyset$.
\end{definition}
Note that in order for an antichain $Q\subseteq {\ints}^d$ to be 
order-generic we must have $d\geq k!$. We will see that the 
structure of the antichain $S_u(Q_*(-\tilde{1},\tilde{b}))$ of ${\ints}_0^d$ 
when $Q$ is order-generic is nice enough for some specific enumerations.

The following slight reformulation follows directly
from Lemma~\ref{lmm:p1-pk-box} and Observation~\ref{obs:p1-pk-box}.
\begin{corollary}
\label{cor:p1-pk-box}
For ${\cal{I}} = \{i_1,\ldots i_k\}\subseteq [d]$ containing distinct
elements and 
\[
D = D(\tilde{p}_1,\ldots,\tilde{p}_k,\tilde{b}(1),\ldots,
\widehat{\tilde{b}(i_1)},\ldots,\widehat{\tilde{b}(i_k)},\ldots,\tilde{b}(d))
\]
we have 
\[
S_u(D) = \{(p_{\alpha(1)\/i_1}+1)\tilde{e}_{i_1} 
+ \cdots 
+ (p_{\alpha(k)\/i_k}+1)\tilde{e}_{i_k}\},
\]
if there is a unique permutation $\alpha\in S_k$ with
$p_{\alpha(h)\/i_{h}} < p_{\alpha(1)\/i_{h}},
\ldots, \widehat{p_{\alpha(h)\/i_{h}}},\ldots, p_{\alpha(k)\/i_{h}}$
for each $h\in[k]$ and $S_u(D) = \emptyset$ otherwise.
\end{corollary}
For $k\in\nats$ we have that $S_u(Q_*(-\tilde{1},\tilde{b})$, the set
of ${\nats}_0^d$ corresponding to the generators of our desired monomial
ideal, is the antichain formed by the maximal elements of all
the $\binom{d+k}{d} = \binom{d+k}{k}$ possible singletons 
$S_u(D)$, where $D\subseteq Q_*(-\tilde{1},\tilde{b})$
contains exactly $d$ elements, and hence there are $k$
types of points in $S_u(Q_*(-\tilde{1},\tilde{b})$:
for each $\ell\in[k]$ the ones that are formed
by a $d$-set $D$ that contains exactly $\ell$ points
from $Q = \{\tilde{p}_1,\ldots,\tilde{p}_k\}$. Since our 
antichain is order-generic each mentioned type occurs 
in $S_u(Q_*(-\tilde{1},\tilde{b})$. To effectively
list of these points we make the following convention.
For each $a\in [k]$ and a subset $C\subseteq [k]$
we let
\begin{equation}
\label{eqn:i-C}
B_{a}(C) := \left\{i\in[d] : p_{a\/i} < p_{b\/i}
\mbox{ for all } b\in C\setminus \{a\} \right\}.
\end{equation}
For $C\subseteq C'\subseteq [k]$ we clearly have 
$B_{a}(C')\subseteq B_{a}(C)$. With this convention
the sets of the mentioned type $\ell\in[k]$, that contain
those points formed by $d$-set $D$ containing $\ell$ points from $Q$, 
are those labeled by a $C\subseteq [k]$ with $|C| = \ell$, 
and are of the form
\begin{equation}
\label{eqn:P_C}
P_C = \left\{\sum_{t\in C}(p_{t\/i_t} + 1)\tilde{e}_{i_t} :
(i_t)_{t\in C}\in \prod_{t\in C}B_t(C)\setminus
\bigcup_{C'\supset C}\prod_{t\in C}B_t(C')\right\}.
\end{equation}
In particular, for $\ell = k$ we have
\begin{equation}
  \label{eqn:P_1...k} 
P_{[k]} = \{(p_{1\/i_1}+1)\tilde{e}_{i_1} + \cdots +
  (p_{k\/i_k}+1)\tilde{e}_{i_k} :
  (i_1,\ldots,i_k)\in B_1([k])\times\cdots\times B_k([k])\}.
\end{equation}
and for $\ell = k-1$ we obtain the sets labeled by
$[\hat{a}] := \{1,\ldots,\hat{a},\ldots,k\}$ for each $a\in[k]$ 
as follows
\begin{eqnarray}
\lefteqn{P_{[\hat{a}]} = \{(p_{1\/i_1}+1)\tilde{e}_{i_1} + \cdots +
  (p_{k\/i_k}+1)\tilde{e}_{i_k}
  : (i_1,\ldots,\widehat{i_a},\ldots,i_k) } \nonumber \\
& \in & B_1([\hat{a}])\times\cdots\times\widehat{B_a([\hat{a}])}\times\cdots
\times B_k([\hat{a}]) \setminus B_1([k])\times\cdots\times\widehat{B_a([k])}
\times\cdots\times B_k([k])\}\}. \label{eqn:Pi(1hatjk)}
\end{eqnarray}
Note that when $\ell = 1$, so our $d$-set contains a single element
$\tilde{p}_a$ from $Q$, then we have by this construction that
\[
P_{\{a\}} = \{(p_{a\/i} + 1)\tilde{e}_i :
p_{a\/i}\geq p_{b\/i}\mbox{ for all } b\neq a\}.
\]
\begin{rmks*}\ 
\begin{enumerate}[label=(\roman*)]
\item Note that the sets $P_C$ are constructed directly in terms of 
the sets $B_a(C)$ from (\ref{eqn:i-C}), each which is union
of sets $A_{\omega}$ from (\ref{eqn:partk}).
\item The fact that
$Q$ is order-generic ensures that the sets $B_a(C)$ are all
non-empty, which in return will imply that each $P_C$ from 
(\ref{eqn:P_C}) does not contain comparable elements, 
and is therefore an antichain, as we will see below.
\end{enumerate}
\end{rmks*}

We clearly have $S_u(D(Q_*(-\tilde{1},\tilde{b})))
\subseteq \bigcup_{\emptyset\neq C\subseteq [k]}P_C$,
but we have not verified that this union of these sets $P_C$
is actually an antichain. By the following two lemmas
we will see that this is indeed the case.
\begin{lemma}
\label{lmm:comparable-l}
Let $Q = \{\tilde{p}_1,\ldots,\tilde{p}_k\}$ be an order-generic antichain 
in ${\ints}^d$. If
$\tilde{r} = (p_{a_1\/i_1} + 1)\tilde{e}_{i_1} + \cdots
+ (p_{a_{\ell}\/i_{\ell}} + 1)\tilde{e}_{i_{\ell}}\in P_{\{a_1,\ldots,a_{\ell}\}}$
and 
$\tilde{s} = (p_{b_1\/i_1} + 1)\tilde{e}_{i_1} + \cdots
+ (p_{b_{\ell}\/i_{\ell}} + 1)\tilde{e}_{i_{\ell}}\in P_{\{b_1,\ldots,b_{\ell}\}}$,
then $\tilde{r}$ and $\tilde{s}$ are comparable in ${\ints}^d$
if and only if $\tilde{r} = \tilde{s}$.
\end{lemma}
\begin{proof}
Suppose $\tilde{r} < \tilde{s}$ so $\tilde{r}\neq\tilde{s}$. In this
case there is a coordinate, say $i_1$ for simplicity, with
$p_{a_1\/i_1} < p_{b_1\/i_1}$.

Suppose that $b_1 = a_h \in \{a_1,\ldots,a_{\ell}\}$. We first
note that $h\neq 1$, since if $h=1$, then 
$p_{b_1\/i_1} = p_{a_h\/i_1} = p_{a_1\/i_1}$.  
We then have $p_{a_h\/i_h} = p_{b_1\/i_h} > p_{b_h\/i_h}$ since
$\tilde{s}\in P_{\{b_1,\ldots,b_{\ell}\}}$, and so looking at coordinates
$i_1$ and $i_h$ we see that $\tilde{r}$ and $\tilde{s}$ are in this
case incomparable, which is a contradiction. Therefore 
we have that $b_1 \not\in \{a_1,\ldots,a_{\ell}\}$.
 
By definitions of $B_a(C)$ in (\ref{eqn:i-C}) we then have 
$i_j\in B_{a_j}(\{b_1,a_1,\ldots,a_{\ell}\})$ for $1\leq j\leq \ell$
and therefore 
$(i_1,\ldots,i_{\ell})\in B_{a_1}(C')\times\cdots\times B_{a_{\ell}}(C')$
for $C' = \{b_1,a_1,\ldots,a_{\ell}\}\supset \{a_1,\ldots,a_{\ell}\}$ 
contradicting $\tilde{r} = \in P_{\{a_1,\ldots,a_{\ell}\}}$. 
Hence, $\tilde{r} = \tilde{s}$ must hold. This completes the proof.
\end{proof}
\begin{rmk*} We can say a bit more than 
in the above proof: since $Q$ is order-generic there is an
$i_{\ell + 1} \in B_{b_1}(\{b_1,a_1,\ldots,a_{\ell}\})$ and so
\[
\tilde{r}' := \tilde{r} + (p_{b_1\/i_{\ell + 1}} + 1)\tilde{e}_{\ell + 1}
\in P_{C'} = P_{\{b_1,a_1,\ldots,a_{\ell}\}},
\]
or is dominated by an element in $P_{C''}$ where 
$\{b_1,a_1,\ldots,a_{\ell}\}\subseteq C''$. 
\end{rmk*}

The claims of Lemma~\ref{lmm:comparable-l} actually hold in a slightly
more general setting: suppose 
$Q = \{\tilde{p}_1,\ldots,\tilde{p}_k\}$ is
an order-generic antichain in ${\ints}^d$. Suppose also that  
$C\subseteq C'\subseteq [k]$, $\tilde{r}\in P_C$, 
$\tilde{s}'\in P_{C'}$ and $\tilde{r}\leq \tilde{s}'$ in ${\ints}^d$.
In this case we can write $C = \{a_1,\ldots,a_{\ell}\}$ and
$C' = \{b_1,\ldots,b_h\}$ where $\ell\leq h$. If  
$\tilde{r} = (p_{a_1\/i_1} + 1)\tilde{e}_{i_1} + \cdots
+ (p_{a_{\ell}\/i_{\ell}} + 1)\tilde{e}_{i_{\ell}}\in P_C$
and 
$\tilde{s}' = (p_{b_1\/i_1} + 1)\tilde{e}_{i_1} + \cdots +
(p_{b_h\/i_h} + 1)\tilde{e}_{i_h}\in P_{C'}$, then,
since $\tilde{r}\leq\tilde{s}'$ there is a projective image
$\tilde{s}$ of $\tilde{s}'$, which we can assume has the form
$\tilde{s} = (p_{b_1\/i_1} + 1)\tilde{e}_{i_1} + \cdots
+ (p_{b_{\ell}\/i_{\ell}} + 1)\tilde{e}_{i_{\ell}}$, such that 
$\tilde{r}\leq\tilde{s}$. As in the proof 
of the previous Lemma~\ref{lmm:comparable-l}, if $\tilde{r}\neq\tilde{s}$ then
there is a coordinate, say $i_1$ for simplicity, with
$p_{a_1\/i_1} < p_{b_1\/i_1}$. Also, since $\tilde{r} < \tilde{s}$ we 
cannot have $b_1\in\{a_1,\ldots,a_{\ell}\}$ and so 
$C' = \{b_1,a_1,\ldots,a_{\ell}\}\subset \{a_1,\ldots,a_{\ell}\}$.
Since $i_j\in B_{a_j}(C')$ for $1\leq j\leq \ell$ we also get in
this case a contradiction to the fact that $\tilde{r}\in P_C$.
We summarize in the following Lemma.
\begin{lemma}
\label{lmm:comparable-l++}
Let $Q = \{\tilde{p}_1,\ldots,\tilde{p}_k\}$ be an order-generic antichain 
in ${\ints}^d$ and suppose $C\subseteq C'\subseteq [k]$, $\tilde{r}\in P_C$ and
$\tilde{s}'\in P_{C'}$. If $\tilde{r}\leq \tilde{s}'$,
then $\tilde{r} = \tilde{s}$ where $\tilde{s}$ is a projective
image of $\tilde{s}'$, and this can only occur when $C' = C$
and $\tilde{s}' = \tilde{s} = \tilde{r}$.
\end{lemma}
By Lemmas~\ref{lmm:comparable-l} and~\ref{lmm:comparable-l++} we have
in particular that the union $\bigcup_{\emptyset\neq C\subseteq [k]}P_C$
is an antichain in ${\ints}^d$. We therefore have the following.
\begin{proposition}
\label{prp:type k}
For an order-generic antichain $Q = \{\tilde{p}_1,\ldots,\tilde{p}_k\}$
in ${\ints}^d$ we have 
\begin{equation}
\label{eqn:union}
S_u(D(Q_*(-\tilde{1},\tilde{b}))) = \bigcup_{\emptyset\neq C\subseteq [k]}P_C,
\end{equation}
where the sets $P_C$ are as in (\ref{eqn:P_C}).
\end{proposition}

\begin{rmks*}\ 
\begin{enumerate}[label=(\roman*)]
\item Note that (\ref{eqn:union}) does
not hold when $Q$ is not order-generic, since the union on 
the right is not in general an antichain. 
\item Although the union on the right in (\ref{eqn:union}) is not, 
in general, disjoint, it does yield an inclusion/exclusion-like formula for
explicit enumeration for an order-generic $Q$.
This will be demonstrated here below for the case $k=3$.
\end{enumerate}
\end{rmks*}

\subsection*{The case when $k=3$}
Here we assume that our order-generic antichain 
$Q = \{\tilde{p}_1,\tilde{p}_2,\tilde{p}_3\}\subseteq {\ints}_0^d$
contains three points. This case, as we will see, is considerably
more complicated than the case $k=2$. 

For distinct $i_1,i_2,i_2\in [d]$ and
\[
D = D(\tilde{p}_1,\tilde{p}_2,\tilde{p}_3,\tilde{b}(1),\ldots,
\widehat{\tilde{b}(i_1)},\ldots,
\widehat{\tilde{b}(i_2)},\ldots,
\widehat{\tilde{b}(i_3)},\ldots,
\tilde{b}(d))
\]
we have by Corollary~\ref{cor:p1-pk-box} that 
\[
S_u(D) = \{(p_{\alpha(1)\/i_1}+1)\tilde{e}_{i_1} 
+ (p_{\alpha(2)\/i_2}+1)\tilde{e}_{i_2} 
+ (p_{\alpha(3)\/i_3}+1)\tilde{e}_{i_3}\},
\]
if there is a unique permutation $\alpha\in S_3$ with
$p_{\alpha(1)\/i_1} < p_{\alpha(2)\/i_1}, p_{\alpha(3)\/i_1}$;
$p_{\alpha(2)\/i_2} < p_{\alpha(1)\/i_2}, p_{\alpha(3)\/i_2}$; and
$p_{\alpha(3)\/i_3} < p_{\alpha(1)\/i_3}, p_{\alpha(2)\/i_3}$
and $S_u(D) = \emptyset$ otherwise.

For $\{u,v,w\} = \{1,2,3\}$ we will write the parts of $[d]$ as defined
in (\ref{eqn:partk}) in the following way:
\begin{eqnarray*}
A_{u,v,w} & = & \{i\in [d] : p_{u\/i} < p_{v\/i} < p_{w\/i}\}, \\
A_{uv,w} & = & \{i\in [d] : p_{u\/i} = p_{v\/i} < p_{w\/i}\}, \\
A_{u,vw} & = & \{i\in [d] : p_{u\/i} < p_{v\/i} = p_{w\/i}\},
\end{eqnarray*}
and 
$A_{123} = \{i\in [d] : p_{1\/i} = p_{2\/i} = p_{3\/i}\}$.
In this way our pseudo-partition of $[d]$ from 
in (\ref{eqn:partk}) becomes:
\begin{eqnarray}
[d] & = & A_{123}\cup A_{12,3}\cup A_{13,2}\cup A_{23,1}\cup 
A_{1,23}\cup A_{2,13}\cup A_{3,12} \nonumber \\ 
& \cup & A_{1,2,3}\cup A_{1,3,2}\cup A_{2,1,3}\cup
A_{2,3,1}\cup A_{3,1,2}\cup A_{3,2,1}. \label{eqn:part3}
\end{eqnarray}
As $S_u(Q_*(-\tilde{1},\tilde{b})$ is the 
antichain formed by the maximal elements of all
the $\binom{d+3}{d} = \binom{d+3}{3}$ possible singletons 
$S_u(D)$, where $D\subseteq Q_*(-\tilde{1},\tilde{b})$
contains exactly $d$ elements, then there are 
three types of points in $S_u(Q_*(-\tilde{1},\tilde{b})$:
those ones obtained from a $d$-set $D$ that contains all
the three points of $Q = \{\tilde{p}_1,\tilde{p}_2,\tilde{p}_3\}$,
those obtained from a $D$ that contain exactly two points
of $Q$, and finally those obtained from a $D$ that contain
exactly one of $\tilde{p}_1$, $\tilde{p}_2$ or $\tilde{p}_3$.

In order to facilitate our notation and presentation we define
the following subsets of $[d]$, where $\{u,v,w\} = \{1,2,3\}$:
\begin{eqnarray*}
B_{\langle u,v\rangle}  & := & A_{uw,v} \cup A_{w,u,v}, \\
B^{\langle u\rangle}    & := & A_{v,w,u} \cup A_{w,v,u} \cup A_{vw,u}, \\ 
B^{\langle u*\rangle}   & := & A_{w,uv} \cup A_{v,uw},
\end{eqnarray*}
that is $B_{\langle u,v\rangle}$ contains all the coordinates
$i\in [d]$ where $p_{u\/i} < p_{v\/i}$ and $p_{u\/i}$ 
is not the sole minimum, $B^{\langle u\rangle}$ contains all the coordinates 
$i\in [d]$ where $p_{u\/i}$ is the sole maximum among $p_{u\/i},p_{v\/i}$ 
and $p_{w\/i}$ and $B^{\langle u*\rangle}$ contains
those coordinates $i$ where $p_{u\/i}$ and exactly one of the other
two, $p_{v\/i}$ or $p_{w\/i}$, are the maximum ones.

By (\ref{eqn:P_1...k}) the set of points in 
$S_u(Q_*(-\tilde{1},\tilde{b})$ that are formed by 
the sets $S_u(D)$ where $D\subseteq Q_*(-\tilde{1},\tilde{b})$
contain all the three points of $Q = \{\tilde{p}_1,\tilde{p}_2,\tilde{p}_3\}$,
is then given by
\begin{equation}
\label{eqn:P123}
P_{[3]} = \{(p_{1\/i_1}+1)\tilde{e}_{i_1} + (p_{2\/i_2}+1)\tilde{e}_{i_2} 
+ (p_{3\/i_3}+1)\tilde{e}_{i_3} : (i_1,i_2,i_3)
\in B_1([3])\times B_2([3])\times B_3([3])\}.
\end{equation}

When $\{u,v,w\} = \{1,2,3\}$ then we have 
\[
 B_u(\{u,v\})\times B_v(\{u,v\})\setminus B_u([3])\times B_v([3])
= B_{\langle u,v\rangle}\times B_{\langle v,u\rangle} 
\cup B_{\langle u,v\rangle}\times B_v([3]) 
\cup B_u([3])\times B_{\langle u,v\rangle}
\]
and hence the set $P_{\hat{w}} = P_{u\/v}$ of points in 
$S_u(Q_*(-\tilde{1},\tilde{b})$ formed by the sets $S_u(D)$ where 
$D$ contains exactly two points $\tilde{p}_u$ and $\tilde{p}_v$ are 
by (\ref{eqn:Pi(1hatjk)}) given by
\begin{equation}
\label{eqn:Puv}
P_{\{u,v\}} = \{(p_{u\/i}+1)\tilde{e}_i + (p_{v\/j}+1)\tilde{e}_j 
: (i,j)\in 
B_{\langle u,v\rangle}\times B_{\langle v,u\rangle} 
\cup B_{\langle u,v\rangle}\times B_v([3]) 
\cup B_u([3])\times B_{\langle u,v\rangle}
\}.
\end{equation}

Lastly, the set of points in $S_u(Q_*(-\tilde{1},\tilde{b})$ formed by the 
sets $S_u(D)$ where $D$ contains exactly one point $\tilde{p}_i\in Q$ 
is given by
\begin{equation}
\label{eqn:Pu}
P_{\{u\}} = \{(p_{u\/i}+1)\tilde{e}_i : 
i \in B^{\langle u\rangle}\cup B^{\langle u*\rangle} \cup A_{123} \}.
\end{equation}
Note that for $\{u,v,w\} = \{1,2,3\}$ we have
\begin{eqnarray*}
&      & \left(B_{\langle u,v\rangle}\times B_{\langle v,u\rangle} 
         \cup B_{\langle u,v\rangle}\times B_v([3]) 
         \cup B_u([3])\times B_{\langle u,v\rangle}\right) \\
& \cap & \left(B_{\langle u,w\rangle}\times B_{\langle w,u\rangle} 
         \cup B_{\langle u,w\rangle}\times B_v([3]\right) 
         \cup B_u([3])\times B_{\langle u,w\rangle}) \\
&  =   & B_{\langle u\rangle}\times A_{vw,u}
\end{eqnarray*}
and hence by Lemma~\ref{lmm:comparable-l} we have the following.
\begin{claim}
\label{clm:Puv-Puw}
Two points $\tilde{r}\in P_{\{u,v\}}$ and $\tilde{s}\in P_{\{u,w\}}$
are comparable in ${\ints}_0^d$ if and only they are equal in which
case $\tilde{r} = \tilde{s} = (p_{u\/i}+1)\tilde{e}_i +(p_{v\/j}+1)\tilde{e}_j$,
where $(i,j) \in B_{\langle u\rangle}\times A_{vw,u}$.
\end{claim}
This claim together with (\ref{eqn:Puv}) makes it possible to enumerate the 
maximal elements in the union $P_{\{u,v\}}\cup P_{\{u,w\}}$ in a direct manner.

By symmetry and the above Claim~\ref{clm:Puv-Puw} we have that
if two points in $P_{\{u,v\}}$ and $P_{\{v,w\}}$ respectively are comparable,
then the must be equal, say to 
$(p_{u\/i}+1)\tilde{e}_{i}+(p_{v\/j}+1)\tilde{e}_{j}$
where $(i,j) \in A_{uw,v}\times B_{\langle v\rangle}$. Since 
$(B_{\langle u\rangle}\times A_{vw,u})\cap (A_{uw,v}\times B_{\langle v\rangle}) 
= \emptyset$
we have, in particular, that no three points in $P_{\{u,v\}}$, $P_{\{u,w\}}$
and $P_{\{v,w\}}$ respectively, are pairwise comparable in ${\ints}_0^d$.

Lastly, consider the sets $P_{\{1\}}$, $P_{\{2\}}$ and $P_{\{3\}}$ 
from (\ref{eqn:Pu}). Since for any distinct $u,v\in \{1,2,3\}$ we have
\[
(B^{\langle u\rangle}\cup B^{\langle u*\rangle} \cup A_{123})\cap
(B^{\langle v\rangle}\cup B^{\langle v*\rangle} \cup A_{123})
= A_{w,uv}\cup A_{123},
\]
then clearly two comparable elements in $P_{\{u\}}$ and $P_{\{v\}}$ respectively
must be equal. Also note that 
$P_{\{1\}}\cap P_{\{2\}} \cap P_{\{3\}} 
= \{(p_{u\/i}+1)\tilde{e}_i : i\in A_{123} \}$.

\begin{convention*}
For an upper case letter $X \in \{A,B\}$ we denote
the cardinality of $X_*$, $X^*$ by the corresponding lower case
boldface letter ${\mathbf{x}}_*$ and ${\mathbf{x}}^*$ respectively, 
so $|A_{12,3}| = {\mathbf{a}}_{12,3}$,
$|B^{\langle 2*\rangle}| = {\mathbf{b}}^{\langle 2*\rangle}$ etc. 
\end{convention*}
With this convention we can now list and
also enumerate the elements in $S_u(Q_*(-\tilde{1},\tilde{b})$ 
by the inclusion/exclusion principle in the following.
\begin{proposition}
\label{prp:type 3}
For an order-generic antichain 
$Q = \{\tilde{p}_1,\tilde{p}_2,\tilde{p}_3\}$ of ${\ints}_0^d$
and the corresponding pseudo-partition of $[d]$
as in (\ref{eqn:part3}),
we have 
\[
S_u(D(Q_*(-\tilde{1},\tilde{b}))) 
= P_{\{1,2,3\}} \cup P_{\{1,2\}} \cup P_{\{1,3\}} \cup P_{\{2,3\}} 
\cup P_{\{1\}} \cup P_{\{2\}} \cup P_{\{3\}},
\]
where $P_{\{1,2,3\}}$, $P_{\{u,v\}}$ and each $P_{\{u\}}$ are as in 
(\ref{eqn:P123}), (\ref{eqn:Puv}) and (\ref{eqn:Pu}) respectively. 

Further, the cardinality is given by
\begin{eqnarray*}
  \lefteqn{|S_u(D(Q_*(-\tilde{1},\tilde{b})))| =
            {\mathbf{b}}_1([3]){\mathbf{b}}_2([3]){\mathbf{b}}_3([3])
          + {\mathbf{b}}_{\langle 1,2\rangle}{\mathbf{b}}_{\langle 2,1\rangle}
          + {\mathbf{b}}_{\langle 1,2\rangle}{\mathbf{b}}_{\langle 2\rangle} 
          + {\mathbf{b}}_{\langle 1\rangle}{\mathbf{b}}_{\langle 2,1\rangle}} \\
& + &       {\mathbf{b}}_{\langle 1,3\rangle}{\mathbf{b}}_{\langle 3,1\rangle}
          + {\mathbf{b}}_{\langle 1,3\rangle}{\mathbf{b}}_{\langle 3\rangle} 
          + {\mathbf{b}}_{\langle 1\rangle}{\mathbf{b}}_{\langle 3,1\rangle} 
          + {\mathbf{b}}_{\langle 2,3\rangle}{\mathbf{b}}_{\langle 3,2\rangle} 
          + {\mathbf{b}}_{\langle 2,3\rangle}{\mathbf{b}}_{\langle 3\rangle}
          + {\mathbf{b}}_{\langle 2\rangle}{\mathbf{b}}_{\langle 3,2\rangle} \\
& - &       ({\mathbf{b}}_{\langle 1\rangle}{\mathbf{a}}_{23,1}
          + {\mathbf{b}}_{\langle 2\rangle}{\mathbf{a}}_{13,2}
          + {\mathbf{b}}_{\langle 3\rangle}{\mathbf{a}}_{12,3}) 
          + {\mathbf{b}}^{\langle 1\rangle} 
          + {\mathbf{b}}^{\langle 2\rangle} 
          + {\mathbf{b}}^{\langle 3\rangle}
          + {\mathbf{a}}_{3,12}
          + {\mathbf{a}}_{2,13} + {\mathbf{a}}_{1,23} + {\mathbf{a}}_{123}.
\end{eqnarray*}
\end{proposition}

\begin{rmk*}
Unlike the presentation in Proposition~\ref{prp:type 2}
and Corollary~\ref{cor:mon-type 2} the presentation
for the antichain $S_u(D(Q_*(-\tilde{1},\tilde{b})))$ as a union in 
Proposition~\ref{prp:type 3} 
is symmetric, however it is not a disjoint union, as that would be 
quite convoluted and confusing.
\end{rmk*}

\begin{corollary}
\label{cor:mon-type 3}
For an order-generic antichain
$Q = \{\tilde{p}_1,\tilde{p}_2,\tilde{p}_3\}$ of ${\ints}_0^d$, the
unique local zero dimensional type 3 monomial ideal $I$ of
$R = K[x_1,\ldots,x_d]$ with
$\doc(I) =
\{\tilde{x}^{\tilde{p}_1}, \tilde{x}^{\tilde{p}_2}, \tilde{x}^{\tilde{p}_3}\}$
is given by
\[
I = (\tilde{x}^{\tilde{r}} : \tilde{r} \in S_u(D(Q_*(-\tilde{1},\tilde{b})))\}
\]
where $S_u(D(Q_*(-\tilde{1},\tilde{b})))$ is as in
Proposition~\ref{prp:type 3}.
\end{corollary}

\begin{rmks*}\ 
\begin{enumerate}[label=(\roman*)]
\item Corollary~\ref{cor:mon-type 3} characterizes
completely those monomial ideals $I$ of $R = K[x_1,\ldots,x_d]$
that make $R/I$ a local Artinian ring of type 3, where
$\doc(I)$ corresponds to an order-generic antichain
in ${\ints}^d$ in a purely
combinatorial way, using only the poset structure of ${\ints}^d$.
\item Note that since $S_u(D(Q_*(-\tilde{1},\tilde{b})))$ is an antichain,
then $I$ as presented in Corollary~\ref{cor:mon-type 3} is exactly
its minimal generation, that is,
$\{\tilde{x}^{\tilde{r}} : \tilde{r} \in S_u(D(Q_*(-\tilde{1},\tilde{b})))\}$
is the unique Gr\"{o}bner basis for $I$.
\end{enumerate}
\end{rmks*}

\section{Socles of general ideals}
\label{sec:general}

We have so far only been interested in the combinatorial properties
of monomial ideals of $R = K[x_1,\ldots,x_d]$ that can described
solely in terms of the monoid $[x_1,\ldots,x_d]$ viewed as a poset,
where the partial order is given by divisibility. As a consequence,
the role of the 
field $K$ has so far not played a major role in our investigations
on monomial ideals. We have at times further
assumed our monomial ideal $I\subseteq R$ to be {\em positive}, that is
$I\subseteq (x_1\cdots x_d)$, or each (generating) monomial 
$\tilde{x}^{\tilde{p}}$ of $I$ having $\tilde{p}\in {\nats}^d$.
Also, we so far only considered the the socle
$\soc(I) = \soc_{\m}(I)$ of $R$ w.r.t.~the specific maximal ideal 
$\m = (x_1,\ldots,x_d)\subseteq R$. 

Just to quickly iterate that we are not gaining anything by allowing
$K$-linear combination when considering monomial ideals, we note
that that for a monomial ideal 
$I = (\tilde{x}^{\tilde{p}_1},\ldots,\tilde{x}^{\tilde{p}_k})\subseteq R$
the set of its generators 
$M = \{\tilde{x}^{\tilde{p}_1},\ldots,\tilde{x}^{\tilde{p}_k}\}$ is always
a Gr\"{o}bner basis for $I$ w.r.t.~any term order on $[x_1,\ldots,x_d]$.
As a result, a fully reduced polynomial $f\in R\setminus \{0\}$ 
is exactly a linear combination of fully reduced monomials w.r.t. $I$,
that is, those monomials that do not reduce to zero in $R/I$.
In particular we have the following.
\begin{observation}
\label{obs:monoff}
If $I\subseteq R = K[x_1,\ldots,x_d]$ is a monomial ideal,
then the socle of $R/I$ w.r.t.~the maximal ideal 
$\m = (x_1,\ldots,x_d)$ is exactly the set of all $K$-linear
combinations of the monomials of the socle,
$\soc(I) = \spn_K(\{\tilde{x}^{\tilde{r}} : 
\tilde{x}^{\tilde{r}}\in \soc(I)\})$.
\end{observation}
In this section we show that 
even in these mentioned seemingly restrictive case, we have managed
to capture many of the interesting combinatorial properties 
of the socle of a general ideal w.r.t.~a maximal ideal.

First we recall some notations and facts from commutative ring theory.
For an ideal $I\subseteq R$ let 
$\sqrt{I} = \{f\in R : f^n \in I\mbox{ for some } n\in\nats\}$
denote the radical of $I$. For a field $K$ and $R = K[x_1,\ldots,x_d]$ then 
$V(I) = \{\tilde{a}\in K^d : f(\tilde{a}) = 0\mbox{ for all } f\in I\}$ is
the affine variety defined by $I$. 
For a subset $U\subseteq K^d$ we let 
$I(U) = \{f\in R : f(\tilde{a}) = 0\mbox{ for all } \tilde{a}\in U\}$
denote the ideal of all polynomials in $R$ that vanish on $U$. 
For any field $K$ and a point $\tilde{a}\in K^d$ then clearly
$(x_1-a_1,\ldots,x_d-a_d)\subseteq R$ is a maximal ideal.
For the converse, we recall an important corollary of Hilbert's 
Nullstellensatz~\cite[p.~85]{Atiyah-Macdonald} 
and~\cite[p.~410]{Aluffi} as the following~\cite[Cor.~2.10, p.~406]{Aluffi}.
\begin{theorem}
\label{thm:HilbertNSS-cor}
If $K$ is an algebraically closed field, then an ideal of 
$K[x_1,\ldots,x_d]$ is maximal if and only if it has the form 
$\m_{\tilde{a}} = (x_1-a_1,\ldots,x_d-a_d)$ for some $\tilde{a}\in K^d$.
\end{theorem}
Recall the following theorem from~\cite[Prop.~1.11]{Atiyah-Macdonald}
on {\em prime avoidance}:
\begin{theorem}
\label{thm:prime-avoidance}
Let $\p_1,\ldots,\p_k$ be prime ideals in a commutative ring.
If $I\subseteq \bigcup_{i=1}^k{\p_i}$ is an ideal, then 
$I\subseteq \p_i$ for some $i$.
\end{theorem}
Since $R = K[x_1,\ldots,x_d]$ 
is Noetherian, every ideal $I\subseteq R$
has a primary decomposition $I = \bigcap_{i=1}^k{\q_i}$ which we can
assume to be minimal (or reduced). For each $i$ let $\p_i = \sqrt{\q_i}$ 
be the associated prime ideal. The set $\{\p_1,\ldots,\p_k\}$ of
associated primes of the ideal $I$ is uniquely determined by $I$ and
is always a finite set.

Suppose $\m\subseteq R$ is a maximal ideal containing the ideal $I$
is such that $\soc_{\m}(I)$ is nontrivial, so $\soc_{\m}(I) = (I:\m)\neq I$.
Since
\[
I\neq (I:\m) = \bigcap_{f\in\m}(I:f),
\]
then for each $f\in\m$ we have $(I:f)\neq I$ as well, and so
by~\cite[Prop.~4.7]{Atiyah-Macdonald} we then have 
\[
\m \subseteq \{f\in R: (I:f)\neq I\} = \bigcup_{i=1}^k\p_i.
\]
By Theorem~\ref{thm:prime-avoidance} above
we therefore have that $\m\subseteq \p_i$ for some $i$, and since
$\m$ is maximal we have $\m = \p_i$. We summarize in the following.
\begin{proposition}
\label{prp:soc-nontriv}
If $R$ is a Noetherian ring, $I\subseteq R$ an ideal and $\m\subseteq R$
a maximal ideal containing $I$ such that $\soc_{\m}(I) = (I:\m)\neq I$,
then $\m$ is one of the finitely many associated prime ideals of $I$.
\end{proposition}
If now $I\subseteq R = K[x_1,\ldots,x_d]$ is a monomial ideal
then every associated prime of $I$ is a {\em face ideal}, or
a {\em Stanley-Reisner ideal}~\cite[p.~19]{Miller-Sturmfels},
that is an ideal of the form $(x_{i_1},\ldots,x_{i_h})$
(See~\cite[Prop.~5.1.3]{MonAlg}.) Hence, if $\m\subseteq R$
is a maximal ideal such that $\soc_{\m}(I)\neq I$, then
by Proposition~\ref{prp:soc-nontriv} $\m$ must be a face ideal
and so we have the following.
\begin{proposition}
\label{prp:mon-max}
If $K$ is a field, $I\subseteq R = K[x_1,\ldots,x_d]$ a monomial ideal
and $\m\subseteq R$ a maximal ideal containing $I$ such that
$\soc_{\m}(I) = (I:\m)\neq I$, then $\m = (x_1,\ldots,x_d)$.
\end{proposition}
If now $K$ is algebraically closed, then by Theorem~\ref{thm:HilbertNSS-cor}
we have a bijective correspondence between points $\tilde{a}\in K^d$
and maximal ideals $\m_{\tilde{a}} = (x_1-a_1,\ldots,x_d-a_d)$ of
$R = K[x_1,\ldots,x_d]$. 
\begin{observation}
\label{obs:Iinm}
Let $R = K[x_1,\ldots,x_d]$ where $K$ is algebraically closed and
$I\subseteq R$ a proper ideal. A maximal ideal $\m$ of $R$ contains
$I$ if and only if $\m = \m_{\tilde{a}}$ for some $\tilde{a}\in V(I)$.
\end{observation}
For an ideal $I\subseteq R = K[x_1,\ldots,x_d]$ where $K$ is algebraically
closed and $\m_{\tilde{a}}$ a maximal ideal containing $I$ let
$\soc_{\tilde{a}}(I) := \soc_{\m_{\tilde{a}}}(I) = (I:\m_{\tilde{a}})$
denote the socle of $I$ w.r.t.~the maximal ideal $\m_{\tilde{a}}$.
By Proposition~\ref{prp:soc-nontriv} the following.
\begin{corollary}
\label{cor:soc-nontriv}
If $K$ is an algebraically closed field and 
$I\subseteq R = K[x_1,\ldots,x_d]$ is an ideal of $R$,
then for all but finitely many points $\tilde{a}\in V(I)$
the socle $\soc_{\tilde{a}}(I) = I$ is trivial.
\end{corollary}
Also, directly by Proposition~\ref{prp:mon-max} do we have the following
corollary.
\begin{corollary}
\label{cor:only-0}
If $K$ is an algebraically closed field,
$I\subseteq R = K[x_1,\ldots,x_d]$ is a monomial ideal of $R$
and $\soc_{\tilde{a}}(I)\neq I$, 
then $\tilde{a} = \tilde{0}$.
\end{corollary}

As the combinatorial properties we want to investigate can be phrased
in terms of the monoid $[x_1,\ldots,x_d]$ of indeterminates as a poset, 
we obtain virtually the same combinatorial properties of
our monomial ideal $I\subseteq R = K[x_1,\ldots,x_d]$ regardless of what
the field $K$ is. Hence,
we might as well look at the extended ideal
$\overline{I} = I\overline{K}[x_1,\ldots,x_d]$ where $\overline{K}$
is the algebraic closure of $K$. This is so, since 
for any vector space (in particular ideals) $W$ over $K$, then
$B$ is a basis for $W$ as a vector space over $K$ if, and only if,
$B$ is a basis for $W\otimes_K\overline{K}$ as a vector space over
$\overline{K}$. In particular $\overline{R} := \overline{K}[x_1,\ldots,x_d] 
= \overline{K}\otimes_KK[x_1,\ldots,x_d]$,
and in general for any ideal $I\subseteq R$ the corresponding ideal 
$\overline{I}$ of $\overline{R}$ has the form 
$\overline{I} = I\otimes_K\overline{K}$. 
Hence, as we have been almost exclusively focusing
on the bases of the socle of a monomial ideal consisting of monomials
alone, we will not lose
any of the combinatorial poset structure of the bases by tensoring 
with $\overline{K}$, that is, assuming $K$ is algebraically closed.
In that case, as stated earlier, Theorem~\ref{thm:HilbertNSS-cor}
yields a bijective correspondence between points of $K^d$ and maximal
ideals of $R$.

At this point two natural questions arise: 
\begin{enumerate}[label=(\arabic*)]
\item What ideals $I\subseteq R$ have their socles $\soc_{\tilde{a}}(I)$
trivial for all $\tilde{a}\in V(I)$? 
\item Knowing just the socle $\soc_{\tilde{a}}(I)$ of an a priori
unknown ideal $I$ of $R$, can we retrieve the ideal $I$%
?
\end{enumerate}
In general, the second question has a negative answer. However, in
numerous specific cases the answer is positive, for example
if we know $I$ has dimension zero, or if we know the socle
to be trivial. 

We consider the first question for $I = (f)$ a principal ideal 
of $R = K[x_1,\ldots,x_d]$.
Since $R$ is a UFD we can write $f = r_1^{b_1}\cdots r_h^{b_h}$ 
where each $r_i$ is irreducible in $R$. In this case we have 
the following minimal primary decomposition of $I$ as
\[
I = \bigcap_{i=1}^h (r_i^{b_i}),
\]
and the unique set of associated primes $\{\p_1,\ldots,\p_h\}$
where $\p_i = \sqrt{(r_i^{b_i})} = (r_i)$ for each $i$.
For $d\geq 2$ then clearly none of these principal prime ideals
$\p_i$ are maximal ideals of $R$, and hence, by
Proposition~\ref{prp:soc-nontriv}, $\soc_{\tilde{a}}(I)$
must be trivial.
\begin{observation}
If $I\subseteq R = K[x_1,\ldots,x_d]$ is a principal ideal where $K$ is 
algebraically closed and $d\geq 2$, then $\soc_{\tilde{a}}(I) = I$ is trivial
for all $\tilde{a}\in V(I)$.
\end{observation} 
When $I$ is generated by two or more elements from $R$, things
are, needless to say, more involved, and it seems we must apply
the Euclidean algorithm to obtain Gr\"{o}bner bases for the ideal, in order
to conclude something fruitful -- we leave that as a question
in the next section.

\section{Summary and further questions}
\label{sec:some-q}

We briefly discuss some of the main results in this article
and post some some relevant and motivating questions.

Presenting everything in terms of ${\ints}^d$ or ${\nats}_0^d$
provided with the componentwise partial order, we showed in
Theorem~\ref{thm:main-gen} and its slight generalization
Theorem~\ref{thm:main-gen-ab} that a given antichain
$G$ of ${\ints}^d$ (corresponding to a monomial ideal of the
polynomial ring in $d$ variables) can always be retrieved from
the set $S_d(U(G^*))$ alone, or more generally from
$S_d(U(G^*(\tilde{a}.\tilde{b})))$ as stated in
Theorem~\ref{thm:main-gen-ab}. As a consequence, the monomial
ideal $I$ can always be retrieved from a specific derived version of 
$\doc(I)$ and also from a more general derived version of
it socle as stated in Theorem~\ref{thm:main-gen-ab}.

The main results in Section~\ref{sec:up-down} 
are Theorem~\ref{thm:main-gen-dual} and Theorem~\ref{thm:main-gen-ab-dual}
which are duals of the previously mentioned
Theorems~\ref{thm:main-gen} and~\ref{thm:main-gen-ab}.
This shows that any antichain $Q$ in ${\nats}^d$ is indeed the socle
for a monomial ideal whose minimal generators correspond to
$U(S_u(D(Q_*)))$ or more generally to $U(S_u(D(Q_*(\tilde{a},\tilde{b}))))$
in Theorem~\ref{thm:main-gen-ab}. As a consequence, for any set of incomparable
monomials in the monoid $[x_1,\ldots,x_d]$ w.r.t.~divisibility, there
are monomials ideals $I$ with the given set of monomials as 
$\doc(I)$. 

In Section~\ref{sec:zero-monomial} we focused on Artinian monomial ideals
and obtained some structure theorems for zero dimensional local
type $k$ rings $R/I$ where $I$ is a monomial ideal and $k = 1,2$.

In Section~\ref{sec:kgeq3} we obtained some results for zero dimensional
local type $k$ rings $R/I$ when $k\geq 3$ and where we assumeed
$\doc(I)$, or rather the power points of the corresponding
monomials, to be order-generic. This yielded a specific enumeration
of the minimal generators of $I$ in Proposition~\ref{prp:type 3}.

In the last Section~\ref{sec:general} we stated some observations
about socles of general ideals of Noetherian rings 
and argued that a lot of the interesting combinatorics really is
captured by monomial ideals. 

There are still many questions worth considering related
to what has been covered.
\begin{question}
Let $K$ be algebraically closed, and $I\subseteq R = K[x_1,\ldots,x_d]$ 
an ideal such that that $\soc_{\tilde{a}}(I) = I$ is trivial for all 
$\tilde{a}\in V(I)$. Can we
conclude something about the Gr\"{o}bner basis for $I$ w.r.t.~some term order?
\end{question}
\begin{question}
Can we generalize the results to general ideals, using 
Gr\"{o}bner bases in terms of some term orders?
If not, perhaps one can obtain similar results for
toric ideals, since there all relations can be obtained
and described within the monoid $[x_1,\ldots,x_d]\subseteq R$.
\end{question}
\begin{question}
When translating the results in previous sections
into results for polynomial rings $R = K[x_1,\ldots,x_d]$,
or general Noetherian rings for that matter,
and their ideals, can anything be phrased simpler by using
the ring structure instead or in addition to the partial order
on the monomials induced by divisibility?
\end{question}

\providecommand{\bysame}{\leavevmode\hbox to3em{\hrulefill}\thinspace}
\providecommand{\MR}{\relax\ifhmode\unskip\space\fi MR }
\providecommand{\MRhref}[2]{%
  \href{http://www.ams.org/mathscinet-getitem?mr=#1}{#2}
}
\providecommand{\href}[2]{#2}


\begin{thebibliography}{BPS98}

\bibitem[A00]{A000670}
\url{https://oeis.org/A000670}.

\bibitem[AF74]{RingsCatMod}
Frank~W. Anderson and Kent~R. Fuller, \emph{Rings and categories of modules},
  Graduate Texts in Mathematics, vol.~13, Springer-Verlag, New York-Heidelberg,
  1974.

\bibitem[Agn97]{A-outside}
Geir Agnarsson, \emph{The number of outside corners of monomial ideals},
  J.~Pure Appl.~Algebra \textbf{117/118} (1997), 3--21.

\bibitem[Agn00]{A-co-gen}
\bysame, \emph{Co-generators for algebras over fields and commutative
  applications}, Comm.~Algebra \textbf{28} (2000), no.~9, 4071--4087.

\bibitem[AL94]{Loust}
William~W. Adams and Philippe Loustaunau, \emph{An introduction to {G}r\"obner
  bases}, Graduate Studies in Mathematics, vol.~3, Amer.~Math.~Soc.,
  Providence, RI, 1994.

\bibitem[Alu09]{Aluffi}
Paolo Aluffi, \emph{Algebra: chapter 0}, Graduate Studies in Mathematics, vol.
  104, Amer.~Math.~Soc., Providence, RI, 2009.

\bibitem[AM69]{Atiyah-Macdonald}
Michael~Francis Atiyah and Ian~G. Macdonald, \emph{Introduction to commutative
  algebra}, Addison-Wesley, Reading, Mass.-London-Don Mills, Ont., 1969.

\bibitem[Bei93]{Beintema}
Mark~B. Beintema, \emph{A note on {A}rtinian {G}orenstein algebras defined by
  monomials}, Rocky Mountain J.~Math. \textbf{23} (1993), no.~1, 1--3.

\bibitem[BH97]{Bruns-Herzog}
Winfried Bruns and J{\"u}rgen Herzog, \emph{{Cohen}-{Macaulay} rings}, revised
  ed., Cambridge Studies in Advanced Mathematics, no.~39, Cambridge
  Univ.~Press, Cambridge, 1997.

\bibitem[BPS98]{Bayer-Peeva-Sturmfels}
Dave Bayer, Irena Peeva, and Bernd Sturmfels, \emph{Monomial resolutions},
  Math.~Res.~Lett. \textbf{5} (1998), no.~1-2, 31--46.

\bibitem[BW93]{Becker-etal}
Thomas Becker and Volker Weispfenning, \emph{Gr{\"o}bner bases}, Graduate Texts
  in Mathematics, vol. 141, Springer-Verlag, New York, 1993, A computational
  approach to commutative algebra. In cooperation with {H}einz {K}redel.

\bibitem[Eis95]{Eisenbud}
David Eisenbud, \emph{Commutative algebra with a view toward algebraic
  geometry}, Graduate Texts in Mathematics, vol. 150, Springer-Verlag, New
  York, 1995.

\bibitem[GP06]{DiscMath}
Edgar~G. Goodaire and Michael~M. Parmenter, \emph{Discrete mathematics with
  graph theory}, 3rd ed., Pearson Prentice Hall, 2006.

\bibitem[HV06]{Huneke-Vraciu}
Craig Huneke and Adela Vraciu, \emph{Rings that are almost {G}orenstein},
  Pacific J.~Math. \textbf{225} (2006), no.~1, 85--102.

\bibitem[Mac94]{Macaulay}
Francis~Sowerby Macaulay, \emph{The algebraic theory of modular systems},
  Cambridge Mathematical Library, Cambridge Univ.~Press, Cambridge, 1994,
  Revised reprint of the 1916 original.~ With an introduction by {P}aul
  Roberts.

\bibitem[MS05]{Miller-Sturmfels}
Ezra Miller and Bernd Sturmfels, \emph{Combinatorial commutative algebra},
  Graudate Texts in Mathematics, vol. 227, Springer-Verlag, New York, 2005.

\bibitem[OBn]{OBn-wiki}
\url{https://en.wikipedia.org/wiki/Ordered_Bell_number}.

\bibitem[Sca86]{HerbertScarf}
Herbert~E. Scarf, \emph{Neighborhood systems for production sets with
  indivisibilities}, Econometrica \textbf{54} (1986), no.~3, 507--534.

\bibitem[Sti]{Stirling-2}
\url{https://en.wikipedia.org/wiki/Stirling_numbers_of_the_second_kind}.

\bibitem[Vil01]{MonAlg}
Rafael~H. Villarreal, \emph{Monomial algebras}, Monographs and Textbooks in
  Pure and Applied Mathematics, vol. 238, Marcel Dekker, Inc., New York, 2001.

\bibitem[Wol16]{Wol16}
Anna-Rose~G. Wolff, \emph{The survival complex}, {arXiv}:1602.08998 [math.AC],
  2016.

\end{thebibliography}
\end{document}